\documentclass[review,hidelinks,onefignum,onetabnum]{siamart220329}

\DeclareMathOperator{\fl}{fl}
\DeclareMathOperator{\op}{op}
\newcommand{\abs}[1]{\left\vert#1\right\vert}

\newcommand{\norm}[1]{\left\Vert#1\right\Vert}

\nolinenumbers

\usepackage{lipsum}
\usepackage{amsfonts}
\usepackage{graphicx}
\usepackage{epstopdf}
\usepackage{xcolor}
\usepackage{tikz}
\usepackage{listings}
\usepackage{caption}
\usepackage{float}
\usepackage{xspace}
\usepackage{amsmath, amssymb, amsfonts}
\usepackage{booktabs}
\usepackage{array,multirow,makecell}
\usepackage{stmaryrd}
\usepackage{scalerel}
\usepackage{csquotes}
\usepackage{subfig}
\usepackage[symbol]{footmisc}

\usepackage{algorithmic}
\ifpdf
  \DeclareGraphicsExtensions{.eps,.pdf,.png,.jpg}
\else
  \DeclareGraphicsExtensions{.eps}
\fi

\usepackage{diagbox}

\usetikzlibrary{trees,arrows}
\tikzstyle{level 1}=[level distance=30mm, sibling distance=30mm]
\tikzstyle{level 2}=[level distance=30mm, sibling distance=15mm]
\tikzstyle{level 3}=[level distance=20mm]

\newcolumntype{R}[1]{>{\raggedleft\arraybackslash }b{#1}}
\newcolumntype{L}[1]{>{\raggedright\arraybackslash }b{#1}}
\newcolumntype{C}[1]{>{\centering\arraybackslash }b{#1}}

\newsiamremark{remark}{Remark}
\newsiamremark{hypothesis}{Hypothesis}
\crefname{hypothesis}{Hypothesis}{Hypotheses}
\newsiamthm{claim}{Claim}
\newsiamremark{pro}{Proposition}
\newsiamremark{mode}{Model}
\newsiamremark{defn}{Definition}

\headers{Stochastic Rounding variance}{E-M. EL ARAR, D. SOHIER, P. de Oliveira Castro and E. PETIT}
\title{Bounds on Non-linear Errors for Variance Computation with Stochastic Rounding}
 
\author{E-M. EL ARAR,\thanks{Université Paris-Saclay, UVSQ, Li-PaRAD, Saint-Quentin en Yvelines, France
  (\email{el-mehdi.el-arar@uvsq.fr},\email{ devan.sohier@uvsq.fr}, \email{pablo.oliveira@uvsq.fr}).}
  \and D. SOHIER\footnotemark[2]  \and P. de Oliveira Castro\footnotemark[2]
\and E. PETIT\thanks{Intel Corp
  (\email{eric.petit@intel.com}).}}

\usepackage{amsopn}

\ifpdf
\hypersetup{
  pdftitle=Bounds on Non-linear Errors for Variance Computation with Stochastic Rounding
  pdfauthor={E-M. EL ARAR, D. SOHIER, P. de Oliveira Castro and E. PETIT}
}
\fi

\begin{document}

\maketitle

\begin{abstract}
    The main objective of this work is to investigate non-linear errors and pairwise summation using stochastic rounding (SR) in variance computation algorithms. We estimate the forward error of computations under SR through two methods: the first is based on a bound of the variance and the Bienaymé–Chebyshev inequality, while the second is based on martingales and the Azuma-Hoeffding inequality. The study shows that for pairwise summation, using SR results in a probabilistic bound of the forward error proportional to $\sqrt{\log(n)}u$ rather than the deterministic bound in $O(\log(n)u)$ when using the default rounding mode.
    We examine two algorithms that compute the variance, called ``textbook” and ``two-pass”, which both exhibit non-linear errors.
    Using the two methods mentioned above, we show that these algorithms' forward errors have probabilistic bounds under SR in $O(\sqrt{n}u)$ instead of $nu$ for the deterministic bounds. We show that this advantage holds using pairwise summation for both textbook and two-pass, with probabilistic bounds of the forward error proportional to $\sqrt{\log(n)}u$. 
\end{abstract}



\section{Introduction}

Stochastic Rounding (SR) mode~\cite{survey} is a probabilistic rounding mode: an inexact computation is rounded to the next smaller or larger floating-point number with probability depending on the distances to those numbers. For several algorithms, such as the inner product~\cite{theo21stocha, arar2022stochastic, ilse} and Horner's rule~\cite{ arar2022stochastic, el2022positive}, SR is unbiased and provides tighter probabilistic bounds of the forward error compared to the deterministic bounds obtained with round-to-nearest (RN)~\cite{norm}. In practice, SR shows higher accuracy than RN for some applications and datasets~\cite{arar2022stochastic}, particularly in low-precision formats such as bfloat-16. Additionally, SR avoids numerical stagnation~\cite{theo21stocha} in different applications such as neural networks~\cite{gupta}, ODEs, and PDEs~\cite{climate}.

Previous theoretical studies of SR error bounds have only considered algorithms in which the numerical error is a linear function of each operation rounding error. Two main methods have been proposed to bound the forward error of linear error algorithms such as summation or inner product computation.  The first, referred to as the BC method in the following, computes the variance of the SR computation and applies Bienaymé–Chebyshev inequality to establish a probabilistic error bound~\cite{arar2022stochastic}. The second, called AH method in the following, is based on martingales and Azuma-Hoeffding inequality~\cite{theo21stocha}.  The two methods are complementary, and each has advantages depending on the size of the problem and the target probabilistic analysis.

Hallman and Ipsen~\cite{hallman2022precision} have studied pairwise summation in the context of SR, showing that the forward error for a sum of $n$ values has a probabilistic bound in $O(\sqrt{\log(n)}u)$ instead $O(\log(n)u)$ for RN. In this paper, we propose a more straightforward method that improves Hallman and Ipsen pairwise summation error bound~\cite{hallman2022precision}.

In $1983$, Chan, Golub, and LeVeque proved deterministic error bounds~\cite{chan1983algorithms} for different algorithms computing the variance of a sample of $n$ data points. These algorithms have non-linear errors due to the presence of squaring in the computation.  In this paper, we prove SR forward error bounds for the ``textbook" and ``two-pass" algorithms with recursive and pairwise summation studied by Chan, Golub, and LeVeque. To the best of our knowledge, this is the first paper theoretically studying non-linear problems with SR. 
We extend previous BC and AH methods to the non-linear variance computation by carefully separating the error terms. 

We first introduce some floating point background and the stochastic rounding mode SR-nearness in Section 2, and recall its main properties that we will use throughout the rest of the paper. 

We analyze the error of pairwise summation under SR-nearness in Section 3, using two methods, AH, and BC. We provide probabilistic bounds for the pairwise summation forward error under SR using two methods, the BC and AH methods. Our AH pairwise bound is simpler and at least as tight as the probabilistic bound proposed in~\cite{hallman2022precision}.

We then move to the analysis of variance computations, which, unlike summations, present non-linear errors. This, in particular, materializes in the existence of a bias, which we study in Section 4. We prove that both textbook and two-pass algorithms are biased and that their biases are equal at order 1 but of opposite signs. 


In Section 5, we show that the deterministic bounds of Chan, Golub and LeVeque [2] extend to SR computations by replacing the $n$ in the bounds by $\sqrt{n}$, and introducing a parameter $\lambda$ representing the probability that the bound does not hold. We do it with both BC and AH methods, leading to bounds behaving better when $n\rightarrow \infty$ or $\lambda\rightarrow0$ respectively, and propose an extension DM of the AH method based on a Doob-Meyer decomposition, which allows to better account for the bias and provides a new tool for SR analysis of non-linear errors. 

We then prove that using pairwise summation in variance computations gives bounds in $\sqrt{\log(n)}$ in Section 6. 
We finally compare the obtained bounds by algorithm (textbook or two-pass) and method (deterministic, BC, AH, DM), and discuss the advantages of each in different situations in Section 7.

\section{Notations and definitions}\label{sec:back}
\subsection{Notations}
In this paper, for an integer $n$ and a vector $x\in \mathbb{R}^n$,  we denote by
\begin{itemize}
    \item $\norm{x}_1 = \sum_{i=1}^{n} \abs{x_i}$ and $\norm{x}_2 = \left(\sum_{i=1}^{n} \abs{x_i}^2\right)^\frac12$.
    \item $s= \sum_{i=1}^{n} x_i$ and $m =\frac{1}{n} \sum_{i=1}^{n} x_i =\frac{1}{n} s$.
    \item $\gamma_{n}(u)=(1+u)^{n} -1$.
    \item $\log(n)$ the smallest integer greater than $\log_2(n)$.
\end{itemize}

We adopt the same notations as used in~\cite{chan1983algorithms}. In the following,
the textbook algorithm computes the variance using the formula $y = \sum_{i=1}^{n} x_i^2 - \frac{1}{n}s^2$, while the two-pass algorithm computes the variance using the formula $z = \sum_{i=1}^{n} \left(x_i -m\right)^2$. We do not study the situation with $y=z=0$, in which the relative error is undefined. The statistical variance can be obtained by multiplying $y$ and $z$ by $\frac{1}{n-1}$. Computing $y$ and $z$ exactly results in $y=z$. However, rounding errors disturb the numerical computations and the obtained results $\widehat{y}$ and $\widehat{z}$ are not equal.

The condition number using the 2-norm for the variance computation is defined in~\cite{chan1983algorithms} as $\mathcal{K}_2 = \frac{\norm{x}_2}{\sqrt{y}}$.
We define the condition number using the 1-norm by $\mathcal{K}_1 = \frac{\norm{x}_1}{\sqrt{ny}}$.
Using the Cauchy-Schwarz inequality, $\mathcal{K}_1 \leq \mathcal{K}_2$; $\mathcal{K}_1$ can be lower than $1$ (for instance, consider $n=4$ and $x_1= 1/2$, $x_2=1/4$, $x_3=-x_1$ and $x_4=-x_2$).

Throughout this paper, for a random variable $X$, $E(X)$ denotes its expected value, $V(X)$ denotes its variance and $\sigma(X)$ denotes its standard deviation. The conditional expectation of $X$ given $Y$ is $\mathbb{E}[X / Y]$.

\begin{lemma}\label{lem:proba}
    Let $X$ and $Y$ two random variables, $a,b\in \mathbb{R}_+^*$,  and $ \lambda, \mu \in ]0; 1[$ such that: $\mathbb{P}(\abs{ X }\leq a) \geq 1- \lambda$ and $\mathbb{P}(\abs{ Y }\leq b) \geq 1- \mu$. Then
    \begin{itemize}
        \item $\mathbb{P}(\abs{ X Y }\leq ab) \geq 1- (\lambda +\mu)$,
        \item $\mathbb{P}(\abs{ X }+ \abs{ Y }\leq a + b) \geq 1- (\lambda+\mu)$.
    \end{itemize}
\end{lemma}

\begin{proof}
    
    \begin{align*}
        \mathbb{P}(\abs{X} \abs{Y} \leq ab) & \geq \mathbb{P}( \{\abs{X} \leq a\} \cap  \{\abs{Y} \leq b\})                                                            \\
                                            & =  \mathbb{P}(\abs{X} \leq a) +  \mathbb{P}(\abs{Y} \leq b) -  \mathbb{P}(  \{\abs{X} \leq a\} \cup  \{\abs{Y} \leq b\}) \\
                                            & \geq 1- \lambda + 1- \mu -1 = 1- (\lambda +\mu).
    \end{align*}
    The proof of the second item uses the first point and the following property $\log(ab) = \log(a) + \log(b)$.
\end{proof}

\subsection{Floating-point background}\label{sec:FP_def}

For a given basis $\beta$ and a working precision $p$, a floating-point number is a real $x$ such that $x=  m \times \beta^{e-p}$, where $e$ is the exponent and $m$ is an integer (the significand) such that $\beta^{p-1} \leq \abs{m} < \beta^p$.
In this paper, we don't take into account special floating-point values such as underflow, overflow, denormals, and NaNs. Detailed information on the floating-point format most generally in use in current computer systems is defined in the IEEE-754 standard~\cite{norm}.

Let us denote $\mathcal{F}\subset \mathbb{R}$, the set of floating-point numbers, and $x\in \mathbb{R}$. Upward rounding $\lceil x \rceil $ and downward rounding $\lfloor x \rfloor$ are defined by:
$$ \lceil x \rceil=\min\{y\in \mathcal{F} : y \geq x\},  \quad \lfloor x \rfloor=\max\{y\in \mathcal{F} : y \leq x\},
$$
by definition, $\lfloor x \rfloor  \leq x \leq \lceil x \rceil$, with equality if and only if $x \in\mathcal{F}$.
The floating-point approximation of a real number $x\ne 0$ is one of $\lfloor x\rfloor$ or $\lceil x\rceil$:
\begin{equation}
    \fl(x) =x(1+\delta), \label{fl(x)}
\end{equation}
where $\delta = \frac{ \fl(x) - x}{x}$ is the relative error: $\abs{ \delta }\leq \beta^{1-p}$.
In the following, we use the same notation as~\cite{theo21stocha, ilse} $u=\beta^{1-p}$. IEEE-754 mode RN (round to nearest, ties to even) has the stronger property that $\abs{ \delta }\leq\frac12\beta^{1-p}=\frac12u$. In many works focusing on IEEE-754 RN, $u$ is chosen instead to be $\frac12\beta^{1-p}$.

For $x, y\in\mathcal F$, the considered rounding modes verify  $\fl(x\op y)\in\{\lfloor x\op y\rfloor, \lceil x\op y\rceil\}$ for $\op\in\{+, -, *, /\}$. Moreover, for IEEE-754 rounding modes~\cite{norm} and stochastic rounding~\cite{theo21stocha} the error in one operation is bounded:
\begin{equation}
    \fl(x \op y) = (x \op y)(1+\delta), \; \abs{ \delta }\leq u; \label{fl(xopy)}
\end{equation}
specifically for RN we have $\abs{ \delta }\leq \frac12 u$.

In this paper, we investigate asymptotic results for a problem of size $n$ and precision $u$; $nu \ll 1$ means $ n \rightarrow \infty$, $ u\rightarrow 0$ and $nu \rightarrow 0$.

\subsection{Stochastic rounding}
\label{sec:SR}

Throughout this paper, $\widehat x = \fl(x)$ is the approximation of the real number $x$ under stochastic rounding.
For $x\in \mathbb{R}\setminus \mathcal{F}$, we consider the following stochastic rounding mode, called SR-nearness:

\noindent\begin{minipage}{.52\linewidth}

\begin{align*}
    \fl(x) & = \left\{
    \begin{array}{cl}
        \lceil x \rceil   & \text{with probability $p(x)$,} \\
        \lfloor x \rfloor & \text{with probability $1-p(x)$.}
    \end{array}
    \right.
\end{align*}
\end{minipage}%
\begin{minipage}{.475\linewidth}
\begin{figure}
    \centering
    \begin{tikzpicture}[xscale=5]
        \draw (0,0) -- (1,0);
        \draw[shift={(0,0)},color=black] (0pt,0pt) -- (0pt, 2pt) node[below] {$\lfloor x \rfloor$};
        \draw[shift={(1,0)},color=black] (0pt,0pt) -- (0pt, 2pt) node[below] {$\lceil x \rceil$};
        \draw[shift={(.3,0)},color=black] (0pt,0pt) -- (0pt, 2pt) node[below] {$x$};
        \draw (0,0) .. controls (.15,.4) .. (.3,0)  (0.15,0.5) node {$1-p(x)$};
        \draw (.3,0) .. controls (.65,.7) .. (1,0) (0.65,0.8) node {$p(x)$} ;
    \end{tikzpicture}
    \caption{\textbf{SR-nearness}.}
\end{figure}
\end{minipage}
\\
where  $p(x)= \frac{x - \lfloor x \rfloor}{\lceil x \rceil - \lfloor x \rfloor }$. The rounding SR-nearness  mode is unbiased
\begin{align*}
E(\widehat x) & = p(x)\lceil x \rceil +(1-p(x))\lfloor x \rfloor \\
& = p(x)(\lceil x \rceil - \lfloor x \rfloor) + \lfloor x \rfloor =x.
\end{align*}

In general, under SR-nearness, the error terms in algorithms appear as a sequence of random variables. The following lemma has been proven in~\cite[lem 5.2]{theo21stocha} and shows that this sequence is mean independent when considering operations satisfying the standard model (ie. the result is calculated exactly on the basis of its input, and then rounded according to SR).

\begin{lemma}\label{meanindp}
   Consider a sequence of elementary operations $c_k \leftarrow a_k \op_k b_k$ for $k \geq 1$, with $\op_k$ satisfying the standard model and $\delta_k$ the error of the $k^{\text{th}}$ operation, that is to say, $\widehat c_k = (\widehat{a}_k \op_k \widehat b_k) (1+\delta_k)$.
    The $\delta_k$ are random variables with mean zero and $(\delta_1, \delta_2, \ldots)$ is mean independent, i.e., $ \forall k \geq 2, \mathbb{E}[\delta_k \mid \delta_1,\ldots , \delta_{k-1}] = \mathbb{E}(\delta_k)$.
\end{lemma}

\section{Pairwise summation}
\label{sec:parwise}
It is known that the accumulator implementation of a sum of $n$ numbers $s=\sum_{i=1}^n x_i$ using a binary tree leads to a deterministic error bound in $O(\log(n)u)$.
In this section, we investigate the forward error made by the pairwise summation under SR-nearness.

For the AH method, we construct a martingale straight from the tree levels and then use Azuma-Hoeffding inequality. This technique has the advantage of building a martingale from the entire tree.
For the BC method, we use~\cite[lem 3.1]{arar2022stochastic} and Bienaymé–Chebyshev inequality. 
Both methods show $O(\sqrt{\log(n)}u)$ probabilistic bounds on the forward error. These bounds are simpler and more intuitive than the bounds in~\cite{hallman2022precision}.

Considering $h$ the height of the summation tree, if $2^{h-1} < n < 2^h$, we set the $2^h-n$ absent inputs to zero. Without loss of generality, let us then assume that $n=2^h$.  Denote $S_i^0 = x_i$ and $S_i^k = S_{2i-1}^{k-1} + S_{2i}^{k-1}$ for all $1 \leq i \leq 2^{h-k}$ and  $1\leq k \leq h$. We have
$$
    S_l^k = \sum_{i=(l-1)2^k +1}^{l2^k} x_i \quad \text{and} \quad S_1^h = \sum_{i=1}^{2^h} x_i = s.$$
Let $\widehat{S}_i^0 = S_i^0$ and $\widehat{S}_i^k = (\widehat{S}_{2i-1}^{k-1} + \widehat{S}_{2i}^{k-1})(1+\delta_i^k)$ for all $1\leq i \leq 2^{h-k}$ and $1\leq k \leq h$. We have
$
    \widehat{S}_l^k = \sum_{i=(l-1)2^k +1}^{l2^k} x_i\prod_{j=1}^k(1+\delta_{\lceil \frac{i}{2^{j}} \rceil}^j).
$

In particular
\begin{equation}
    \label{eq:pairwise}
    \widehat{S}_1^h = \sum_{i=1}^{2^h} x_i\prod_{j=1}^h(1+\delta_{\lceil \frac{i}{2^{j}} \rceil}^j) =  \sum_{i=1}^{2^h} x_i \psi_i \ \text{with} \ \psi_i = \prod_{j=1}^h(1+\delta_{\lceil \frac{i}{2^{j}} \rceil}^j).
\end{equation}

As mentioned in Section~\ref{sec:FP_def}, we compare the asymptotic behavior of the forward error bounds. El Arar et al~\cite{arar2022stochastic} have introduced a new approach based on a bound of the variance and Bienaymé–Chebyshev inequality to obtain probabilistic bounds of the forward error and applied it to Horner's rule. These bounds have the advantage of being closer to the forward error for a large $n$ and a fixed probability than the ones based on the Azuma-Hoeffding inequality. At the same time, Higham and Mary~\cite{theo21stocha} and Ipsen and Zhou~\cite{ilse} used martingales and the Azuma-Hoeffding inequality to obtain probabilistic bounds of the forward error. BC bounds prove better than AH asymptotically in $n$, while AH outperforms BC for $\lambda \rightarrow 0 $.
In the following, we present these two methods and show that SR benefits extend to pairwise summation. In particular, our probabilistic bounds are lower than any deterministic ones (at the expense of introducing a probability that they do not hold).

\subsection{BC method}
Let us recall the lemma that bounds the variance of an error product $\varphi = \prod_{k=1}^n (1+\delta_k)$ under SR-nearness. A general expression of this lemma can be found in~\cite[Lemma 3.1]{arar2022stochastic}.

\begin{lemma}(from~\cite[Lemma 3.1]{arar2022stochastic})\label{model}
Under SR-nearness $\varphi$ satisfies 
\begin{enumerate}
    \item $E(\varphi) = 1$.  
    \item $V(\varphi) \leq \gamma_n(u^2) $,
\end{enumerate}
where $\gamma_n(u^2)= (1+u^2)^{n}-1 \approx \exp{(nu^2)} -1 = nu^2 + O(n^2 u^4)$ for $nu^2 \ll 1$.
\end{lemma}
This lemma has been used to study the inner product and Horner's algorithm in~\cite{arar2022stochastic}. For the pairwise summation (provided that $\sum x_i\neq 0$, because we are considering a relative error that cannot be defined if the result is 0), we have
\begin{theorem}
    \label{parwise:var-thm}
    For all $0 < \lambda <1$, the computed $\widehat{S}_1^h $ satisfies under SR-nearness
    \begin{equation}
        \label{pairwise:varince-bound}
        \frac{\abs{\widehat{S}_1^h - S_1^h}}{\abs{S_1^h}} \leq  \kappa \sqrt{\gamma_{\log(n)}(u^2)/\lambda},
    \end{equation}
    with probability at least $1-\lambda$, where $\kappa =\frac{\sum_{i=1}^{n} \abs{ x_i}}{\abs{ \sum_{i=1}^{n}  x_i}}$ is the condition number using the 1-norm of the sum of the $x_i$.
\end{theorem}

\begin{proof}
    By expectation linearity, $E(\widehat{S}_1^h ) = \sum_{i=1}^{2^h}  x_i E(\psi_i)$. Lemma~\ref{model} shows that for all $1\leq i\leq 2^h$, $E(\psi_i) =1$ and $V(\psi_i) \leq \gamma_{h}(u^2)$. It follows that, $E(\widehat{S}_1^h ) = S_1^h$ and $V(\widehat{S}_1^h ) \leq \left(\sum_{i=1}^{2^h} \abs{x_i} \sqrt{V(\psi_i)}\right)^2 \leq  \norm{x}_1^2 \gamma_{h}(u^2)$. the Bienaymé–Chebyshev inequality implies
    $ \mathbb{P}\left(\abs{\widehat{S}_1^h - E(\widehat{S}_1^h) } \leq    \sqrt{V(\widehat{S}_1^h)/\lambda}  \right) \geq 1-\lambda$. Thus,  with probability at least $1-\lambda$,
    \begin{align*}
        \frac{\abs{\widehat{S}_1^h - S_1^h}}{\abs{S_1^h}} & \leq \frac{1}{\abs{S_1^h}} \sqrt{V(\widehat{S}_1^h)/\lambda}
        \leq  \frac{\norm{x}_1}{\abs{S_1^h}}\sqrt{\gamma_{h}(u^2)/\lambda} =\kappa \sqrt{\gamma_{h}(u^2)/\lambda}.
    \end{align*}
    Since $h=\log(n)$, we have with probability at least $1-\lambda$,
    \begin{align*}
        \frac{\abs{\widehat{S}_1^h - S_1^h}}{\abs{S_1^h}} \leq \kappa \sqrt{\gamma_{\log(n)}(u^2)/\lambda}.
    \end{align*}
    
\end{proof}

\subsection{AH method}
This method uses martingales and then applies the Azuma-Hoeffding inequality for a martingale~\cite{Azu67, hoef63}.

\begin{defn}
    \label{def:martingale}
    A sequence of random variables $(M_0,\ldots,M_n)$ is a martingale with respect to the sequence $X_0,\ldots,X_n$ if, for all $k,$
    \begin{itemize}
        \item $M_k$ is a function of $X_0, \ldots, X_k$,
        \item $E(\abs{ M_k }) < \infty,$ and
        \item $E[M_k / X_0,\ldots, X_{k-1}]=M_{k-1}$.
    \end{itemize}
    If $E[M_k /  X_0,\ldots,X_{k-1}] \geq M_{k-1}$, $(M_0,\ldots,M_n)$ is called submartingale.
\end{defn}

\begin{lemma}[Azuma-Hoeffding inequality]
    \label{hoeffding}
    Let $(M_0,\ldots,M_n)$ be a martingale with respect to a sequence $X_0,\ldots,X_n.$ We assume that there exist $a_k<b_k$ such that $a_k \leq M_k - M_{k-1} \leq  b_k$ for $k\in\{ 1,\ldots,n\}$. Then, for any $A  > 0$,
    $$ \mathbb{P}(\abs{ M_n - M_0 }\geq A) \leq 2 \exp \left(
        -\frac{2A^2}{\sum_{k=1}^n(b_k-a_k)^2}
        \right).
    $$
    In the particular case $a_k=-b_k$ and $\lambda = 2 \exp \left(
        -\frac{A^2}{2\sum_{k=1}^n b_k^2} \right) $ we have
    $$ \mathbb{P}\left( \abs{ M_n - M_0 }\leq \sqrt{\sum_{k=1}^n b_k^2} \sqrt{2 \ln (2 / \lambda)} \right) \geq 1- \lambda,
    $$
    where $0< \lambda <1$.
\end{lemma}

\begin{theorem}
    \label{pairwise:mart-thm}
    For all $0 < \lambda <1$, the computed $\widehat{S}_1^h $ satisfies under SR-nearness
    \begin{equation}
        \label{pairwise:martingale-bound}
        \frac{\abs{\widehat{S}_1^h - S_1^h}}{\abs{S_1^h}} \leq \kappa \sqrt{u\gamma_{2\lceil \log(n) \rceil}(u)} \sqrt{
            \ln(2/ \lambda)},
    \end{equation}
    with probability at least $1-\lambda$.
\end{theorem}

\begin{proof}
    Let us denote for $k > 0, \ M_k = \sum_{i=1}^{2^{h-k}} \widehat{S}_i^k - S_i^k $ and $M_0=0$. Then, $M_h = \widehat{S}_1^h - S_1^h $ and $M_k= M_{k-1} + \sum_{i=1}^{2^{h-k}}(\widehat{S}_{2i-1}^{k-1} + \widehat{S}_{2i}^{k-1})\delta_i^k$.  The $\delta_k$ are mean independent, therefore $M_0,\ldots,M_h$ form a martingale with respect to $\{\delta_i^k, \  1\leq i \leq 2^{h-k}, \ 1\leq k \leq h-1\}$. Moreover, Equation~(\ref{eq:pairwise}) yields
    \begin{align*}
        \abs{M_k - M_{k-1}} & \leq  \sum_{i=1}^{2^{h-k}} \abs{(\widehat{S}_{2i-1}^{k-1} + \widehat{S}_{2i}^{k-1})\delta_i^k}
        \leq u  \sum_{i=1}^{2^{h-k}} \abs{\widehat{S}_{2i-1}^{k-1} + \widehat{S}_{2i}^{k-1}}                                 \\
                            & \leq u (1+u)^{k-1} \sum_{i=1}^{2^{h-k}} \abs{ \sum_{m=2^{k-1} (2i-2)+1}^{2^{k-1}(2i-1)} x_m + \sum_{m=2^{k-1} (2i-1)+1}^{2^{k-1}(2i)} x_m  }\\
                            & \leq u (1+u)^{k-1} \sum_{i=1}^{2^{h-k}} \sum_{m=2^{k} (i-1)+1}^{2^{k}i} 
 \abs{x_m} = u (1+u)^{k-1} \sum_{i=1}^{2^{h}} \abs{x_m} \\
                            & = u (1+u)^{k-1} \norm{x}_1.
    \end{align*}
    Denote $C_k =u (1+u)^{k-1} \norm{x}_1  $, Azuma-Hoeffding inequality implies that with probability at least $1-\lambda$,
    $
        \abs{M_h} \leq \sqrt{\sum_{k=1}^h C_k^2} \sqrt{2\ln(2/ \lambda)}
    $. Now
    \begin{align*}
        \sum_{k=1}^h C_k^2 = u^2 \norm{x}_1^2 \sum_{k=1}^h (1+u)^{2(k-1)}
        = u^2 \norm{x}_1^2 \frac{(1+u)^{2h}-1}{(1+u)^2-1} =  u \norm{x}_1^2 \frac{\gamma_{2h}(u)}{u+2}.
    \end{align*}
    Since, $\frac{u}{u+2} \leq \frac{u}{2}$ and $h=\lceil \log(n) \rceil$, we have
    $ \abs{M_h} \leq \norm{x}_1 \sqrt{u \frac{\gamma_{2 \lceil \log(n) \rceil}(u)}{2}} \sqrt{2\ln(2/ \lambda)}$. Finally
    \begin{align*}
        \frac{\abs{\widehat{S}_1^h - S_1^h}}{\abs{S_1^h}} \leq \kappa  \sqrt{u\gamma_{2\lceil \log(n) \rceil}(u)} \sqrt{\ln(2/ \lambda)},
    \end{align*}
    with probability at least $1-\lambda$.
\end{proof}

\noindent\textbf{Comparison with Hallman and Ipsen's pairwise bound~\cite{hallman2022precision}.}
The probabilistic bound proposed in~\cite[cor, 2.14]{hallman2022precision} to the pairwise summation forward error is
\begin{equation}
\label{ipsen:bound}
     \frac{\abs{\widehat{S}_1^h - S_1^h}}{\abs{S_1^h}} \leq \kappa  u\sqrt{h} \sqrt{2\ln(2/\delta)} (1+\phi_{n,h,\eta}), 
\end{equation}
with probability at least $1-(\eta +\delta)$, where $h$  is the height of the computational tree and $\phi_{n,h,\eta} \equiv \lambda_{n,\eta} \sqrt{2h} u \exp{\left( \lambda_{n,\eta}^2 h u^2\right)}$ with $ \lambda_{n,\eta} \equiv \sqrt{2\ln(2n/\eta)}$.

Bound~(\ref{ipsen:bound}) uses two parameters $\delta$ and $\eta$, the sum of which is higher than the probability that the bound does not hold. We could not find any closed form describing the best value for $\eta$ and $\delta$ given that $\delta+\eta=\lambda$, and we doubt that such a form exists. This makes the choice of their values difficult and, to some extent, arbitrary. 

Taking $\delta+\eta=\lambda$ and $h=\log(n)$, comparing the two bounds adds up to comparing $\sqrt{\gamma_{2\log(n)}(u)}\sqrt{\ln(2/(\delta+\eta)}$ and $\sqrt{2uh} \sqrt{\ln(2/\delta)} (1+\phi_{n,h,\eta})$.
$\gamma_{2\log(n)}(u)>2hu$ and $\gamma_{2\log(n)}(u)=2hu+O((hu)^2)$, giving a short advantage to~(\ref{ipsen:bound}) regarding the first factor. However $\ln(2/(\delta+\eta))<\ln(2/\delta)$, and to close the gap, one needs to
take $\eta$ as small as possible, giving the advantage to~(\ref{pairwise:varince-bound}) for the second factor. In the third factor, taking $\eta\rightarrow0$ makes $\phi_{n, h, u}$ grow to $\infty$. Moreover, the term $\phi_{n, h, u}$ is in $O(hu)$, and thus grows more rapidly than the second order terms in bound~(\ref{pairwise:varince-bound}). All in all, the bound established in this paper avoids the use of two parameters with no easy way to choose their values and gives better asymptotic results than the one in~(\ref{ipsen:bound}).

\section{Bias analysis}
\label{sec:bias}

The unbiased nature of SR-nearness extends to various algorithms such as the inner product~\cite{theo21stocha} and Horner's rule~\cite{el2022positive}. Nevertheless, it fails to hold in the general case. In the sequel, we study two algorithms for computing the variance: textbook and two-pass.

\subsection{Textbook algorithm}
\label{sec:bias-textbook}
For $x \in \mathbb{R}^n$, let  $s =  \sum_{i=1}^n x_i$ and
$y=\sum_{i=1}^{n} x_i^2 - \frac{1}{n}s^2$. Using SR-nearness: 
\begin{itemize}
    \item The computed $\widehat{s}$ satisfies $\widehat{s}= \sum_{i=1}^n x_i  \prod_{k=\max(2, i)}^{n}(1+\delta_{k-1}) = \sum_{i=1}^n x_i \phi_i$ with $\phi_i =\prod_{k=\max(2, i)}^{n}(1+\delta_{k-1}) $ for all $1\leq i \leq n$.
    \item The computed $\widehat{y}$ satisfies
    \begin{equation}
    \label{eq:y}
        \widehat{y} = \sum_{i=1}^{n} x_i^2\psi_i - \frac{1}{n} \widehat{s}^2\psi_{n+1},
    \end{equation} 
\end{itemize}
where
$\psi_i=(1+\epsilon_i)\prod_{k=\max(2,i)}^{n+1}(1+\eta_k)$ and $\psi_{n+1}=(1+\epsilon_{n+1}) (1+\eta_{n+1})(1+\theta)$. For all $1\leq i \leq n+1$, $\epsilon_{i}$ and $\eta_{i}$ represent the rounding errors from the products and additions, respectively. $\theta$ represent the error of the division of $ \widehat{s}^2$ by $n$.

\begin{theorem}\label{bias:textbook}
    The quantities $\widehat{s}$ and $\widehat{y}$ satisfy under SR-nearness
    \begin{itemize}
        \item $E(\widehat{s}) = s$,
        \item $E(\widehat{y}) = y - \frac{1}{n}V(\widehat{s})$.
    \end{itemize}
\end{theorem}

\begin{proof}
    The first item can be proved as in the first part of Theorem~\ref{parwise:var-thm} proof. For the second, we have, by expectation linearity, $E(\widehat{y}) =  \sum_{i=1}^{n} x_i^2E(\psi_i) - \frac{1}{n}E(\widehat{s}^2\psi_{n+1})$. Let $\mathbb{F}=\{ \delta_i, \epsilon_j, \eta_k, i\in\{1,\ldots,n-1\}, \ j\in\{1,\ldots,n\}, \ k \in\{2,\ldots,n\} \}$,
    the mean independence property implies that $E(\psi_i)=1$ for all $1 \leq i \leq n$ and $E[\psi_{n+1}/\mathbb{F}]=1$. Therefore, the law of total expectation $E(X)= E(E[X/Y])$  yields
    \begin{align*}
        E(\widehat{s}^2\psi_{n+1}) & = E\left( E[\widehat{s}^2\psi_{n+1}/\mathbb{F}] \right)
        = E\left(\widehat{s}^2 E[\psi_{n+1}/\mathbb{F}] \right)
        = E(\widehat{s}^2)                                                                   \\
                                   & = E(\widehat{s})^2 + V(\widehat{s})
        = s^2 + V(\widehat{s}).
    \end{align*}
    It follows that
    $
        E(\widehat{y}) = \sum_{i=1}^{n} x_i^2 - \frac{1}{n} s^2 -  \frac{1}{n}V(\widehat{s})= y -  \frac{1}{n}V(\widehat{s}).
    $
\end{proof}

\begin{remark}
    Lemma~\ref{model} gives $V(\phi_i) \leq \gamma_{n-1}(u^2)$, so~\cite[thm 3.2]{arar2022stochastic} shows that the bias satisfies
    $$\frac{1}{n}V(\widehat{s}) \leq \frac{1}{n} \norm{x}_1^2 \gamma_{n-1}(u^2) = y \mathcal{K}_1^2 \gamma_{n-1}(u^2).
    $$
    Thus $E(\widehat{y}) \geq y \left( 1- \mathcal{K}_1^2 \gamma_{n-1}(u^2)\right)$.
\end{remark}

\subsection{Two-pass algorithm}
Let $x_1, x_2,\ldots, x_n\in \mathbb{R}$, denote $m = \frac{1}{n} \sum_{i=1}^n x_i$ and $z=\sum_{i=1}^{n} (x_i -m)^2$. Using SR-nearness: 
\begin{itemize}
    \item The computed $\widehat{m}$ satisfies $\widehat{m}=\frac{1}{n} \sum_{i=1}^n x_i \prod_{k=\max(2, i)}^{n+1}(1+\delta_{k-1})$ with $\delta_n$ is the division error by $n$.
    \item The computed $\widehat{z}$ satisfies 
    \begin{equation}
        \label{eq:z}
        \widehat{z} = \sum_{i=1}^{n} (x_i -\widehat{m})^2\psi_i,
    \end{equation}  
\end{itemize}
 where
$\psi_i=(1+\epsilon_i)^2(1+\eta_i) \prod_{k=\max(2,i)}^n(1+\theta_k).$
For all $1 \leq i \leq n$, $\epsilon_i, \eta_i $ and $\theta_i$ represent the rounding errors of subtraction, square, and addition, respectively.
Let us denote $\varphi_i = (1+\epsilon_i)(1+\eta_i) \prod_{k=\max(2,i)}^n(1+\theta_k)$. Then $\psi_i = (1+\epsilon_i)\varphi_i$.

\begin{theorem}
    The quantities $\widehat{m}$ and $\widehat{z}$ satisfy under SR-nearness
    \begin{itemize}
        \item $E(\widehat{m}) = m$,
        \item $E(\widehat{z}) =  z + \frac{1}{n}V(\widehat{s}) + O(n u^2)$,
              where $\frac{1}{n}s = m$.
    \end{itemize}
\end{theorem}

\begin{proof}
    The first item is similar to the first part of Theorem~\ref{parwise:var-thm} proof. For the second, we have by expectation linearity $E(\widehat{z}) =  \sum_{i=1}^{n} E\left( (x_i -\widehat{m})^2\psi_i \right)$. For all $1 \leq i \leq n$, let $\theta_1=0$ and
    $$
        \mathbb{F}_i=\{ \delta_j, \epsilon_k, \eta_l, \theta_l, \  j\in[1;n], \ k\in[1;i], \ \text{and} \ l\in[1;i-1] \}.$$
    The mean independence property implies that $E[(1+\eta_i)\prod_{k=\max(2,i)}^n(1+\theta_k)/ \mathbb{F}_i]=1$.
    Using the law of total expectation, we have
    \begin{align*}
        E\left( (x_i -\widehat{m})^2\psi_i \right) 
         & = E\left( E\left[(x_i -\widehat{m})^2 (1+\epsilon_i)^2 (1+\eta_i)\prod_{k=\max(2,i)}^n(1+\theta_k)/ \mathbb{F}_i\right] \right)   \\
         & = E\left( (x_i -\widehat{m})^2 (1+\epsilon_i)^2 E\left[(1+\eta_i)\prod_{k=\max(2,i)}^n(1+\theta_k)/ \mathbb{F}_i\right] \right)   \\
         & =  E\left( (x_i -\widehat{m})^2 (1+\epsilon_i)^2 \right)= E\left( (x_i -\widehat{m})^2 (1+2 \epsilon_i + \epsilon_i^2) \right) \\
         & = E\left( (x_i -\widehat{m})^2 (1+\epsilon_i^2) \right)  \qquad \text{by Lemma~\ref{meanindp}}                                 \\
         & =  E\left( (x_i -\widehat{m})^2 \right) + E\left( (x_i -\widehat{m})^2 \epsilon_i^2 \right)                                    \\
         & =  (x_i - m)^2 +  V(\widehat{m}) + E\left( (x_i -\widehat{m})^2 \epsilon_i^2 \right).
    \end{align*}
    It follows that
    \begin{align*}
        E(\widehat{z}) & =  \sum_{i=1}^{n} (x_i - m)^2 +  V(\widehat{m}) + E\left( (x_i -\widehat{m})^2 \epsilon_i^2 \right) \\
                       & = z + nV(\widehat{m}) + \sum_{i=1}^{n} E\left( (x_i -\widehat{m})^2 \epsilon_i^2 \right).
    \end{align*}
    Since $\widehat{m} = \frac{1}{n} (1+\delta_{n})\widehat{s} = \frac{1}{n} \widehat{s}+ \frac{1}{n} \delta_{n}\widehat{s}$ and  $\abs{\epsilon_i}^2, \abs{\delta_{n}}^2 \leq u^2$ for all $1 \leq i \leq n$, 
    $$
        V(\widehat{m}) = \frac{1}{n^2}V(\widehat{s}) + O(nu^2) \quad \text{and} \quad \sum_{i=1}^{n} E\left( (x_i -\widehat{m})^2 \epsilon_i^2 \right)= O(nu^2).
    $$
    Therefore
    $
        E(\widehat{z}) =  z + \frac{1}{n}V(\widehat{s}) + O(nu^2).
    $
\end{proof}

Interestingly, these two algorithms under SR have an opposed bias at the first order over $u$.

\begin{remark}
   Lemma~\ref{model} implies that $V(\widehat{m}) \leq \frac{1}{n^2} \norm{x}_1^2 \gamma_{n}(u^2)$. Then
    \begin{align*}
        E(\widehat{z}) & = z + nV(\widehat{m}) + \sum_{i=1}^{n} E\left( (x_i -\widehat{m})^2 \epsilon_i^2 \right)
        \leq z + nV(\widehat{m}) +  u^2 \sum_{i=1}^{n} E \left((x_i -\widehat{m})^2\right)                                    \\
                       & = z + nV(\widehat{m}) + u^2(z + nV(\widehat{m}))\\
        &\leq (1+u^2) (z +\frac{1}{n} \norm{x}_1^2 \gamma_{n}(u^2))                                                
                        = z(1+u^2) (1 +\mathcal{K}_1^2\gamma_{n}(u^2)).
    \end{align*}
\end{remark}

\section{Error analysis for algorithms with non-linear error}
\label{sec:error-non-lin}
This section examines SR for non-linear computations via the previous two algorithms. We use the two methods discussed in the introduction to estimate the forward error. In addition, a new approach based on Doob-Meyer decomposition is proposed for the textbook algorithm.
\subsection{BC method}
\label{sec:variance-method}
This section uses the BC method proposed in~\cite{arar2022stochastic} to provide a probabilistic bound on the forward error of both textbook and two-pass algorithms under SR-nearness.

\subsubsection{Textbook algorithm}
In order to estimate the forward errors of the textbook algorithm, compute
\begin{align*}
    \abs{ \widehat{y} - y } & = \abs{\sum_{i=1}^{n} x_i^2 (\psi_i -1) - \frac{1}{n} (\widehat{s}^2\psi_{n+1} - s^2) }
    \leq \abs{\sum_{i=1}^{n} x_i^2 (\psi_i -1)  } + \frac{1}{n} \abs{ \widehat{s}^2\psi_{n+1} - s^2}                                                                                 \\
                            & = \abs{\sum_{i=1}^{n} x_i^2 (\psi_i -1)  } + \frac{1}{n} \abs{\left((\widehat{s} -s) + s\right)^2\psi_{n+1} - s^2 }                                                \\
                            & \leq \abs{\sum_{i=1}^{n} x_i^2 (\psi_i -1)  } + \frac{1}{n} \left( \abs{ (\widehat{s} -s)^2\psi_{n+1} }+ 2 \abs{ s(\widehat{s} -s)\psi_{n+1} }+ \abs{ s^2(\psi_{n+1}-1)}\right).
\end{align*}
Let $\mathcal{B} = \abs{ (\widehat{s} -s)^2\psi_{n+1} }+ 2 \abs{ s(\widehat{s} -s)\psi_{n+1} }+ \abs{ s^2(\psi_{n+1}-1)}$, the following equation will be used in all proofs of the textbook forward errors
\begin{equation}
    \label{eq:textbook-model}
    \abs{ \widehat{y} - y } = \abs{\sum_{i=1}^{n} x_i^2 (\psi_i -1)  } + \frac{1}{n} \mathcal{B}.
\end{equation}

\begin{remark}
     To handle the non-linearity of errors, the key idea of this approach is to isolate terms of order $1$ in error the errors and then use the previous results on the inner product or summation. 
     This error could be decomposed otherwise. For instance,
     \begin{align*}
          \frac{1}{n} (\widehat{s}^2\psi_{n+1} - s^2) =  \frac{1}{n} (\widehat{s}^2\psi_{n+1} -\widehat{s}  s +\widehat{s}  s - s^2)
         =  \frac{1}{n} \left( \widehat{s}(\widehat{s}\psi_{n+1} -s)   + s (\widehat{s}  - s)\right).
     \end{align*}
     Then, we can apply the same properties on $(\widehat{s}\psi_{n+1} -s)$ and $(\widehat{s}  - s)$. The bounds are different but asymptotically equivalent when $nu \ll 1$. 
\end{remark}

The rounding errors accumulated in the whole process of this algorithm $\phi_i$ and $\psi_i$ satisfy for all $1 \leq i \leq n$,
$$
    \abs{ \phi_i }\leq (1+u)^{n+1-\max(2,i)}, \quad \abs{ \psi_i }\leq (1+u)^{n+3-\max(2,i)} \  \text{and} \ \abs{ \psi_{n+1} }\leq (1+u)^3.
$$
Let us compute the deterministic bound of this algorithm. We have
$$
    \abs{\sum_{i=1}^{n} x_i^2 (\psi_i -1)  } \leq \norm{x}_2^2  \gamma_{n+1}(u).
$$
Since $\abs{s} \leq \norm{x}_1$ and
$\abs{ \widehat{s} -s } = \abs{ \sum_{i=1}^n x_i (\phi_i -1)}\leq  \norm{x}_1 \gamma_{n-1}(u)$, 
\begin{align*}
    \mathcal{B} & \leq (1+u)^3 \norm{x}_1^2 \left(\gamma_{n-1}^2(u)  + 2 \gamma_{n-1}(u) \right) + \norm{x}_1^2 ((1+u)^3 -1) \\
               & = (1+u)^3 \norm{x}_1^2 \left(\gamma_{n-1}^2(u)  + 2 \gamma_{n-1}(u) +1\right) - \norm{x}_1^2               \\
               & = (1+u)^3 \norm{x}_1^2 \left(\gamma_{n-1}(u) +1 \right)^2 - \norm{x}_1^2                                   \\
               & =  \norm{x}_1^2 (1+u)^{2n+1} - \norm{x}_1^2 = \norm{x}_1^2 \gamma_{2n+1}(u).
\end{align*}
Finally
\begin{align}\label{det-bound-text}
    \frac{\abs{ \widehat{y} - y }}{\abs{ y }} \leq  \mathcal{K}_2^2 \gamma_{n+1}(u) + \mathcal{K}_1^2  \gamma_{2n+1}(u).
\end{align}

The following theorem presents a probabilistic bound of the forward error of this algorithm through the BC method.

\begin{theorem}
    \label{thm:var-textbook}
    For all $0 < \lambda <1$, the computed $\widehat{y}$ in Equation~(\ref{eq:y}) satisfies under SR-nearness
    $$  \frac{\abs{ \widehat{y} - y }}{\abs{ y }}
        \leq  \mathcal{K}_2^2 \sqrt{ 2\gamma_{n+1}(u^2)/ \lambda} + \mathcal{K}_1^2 \left( (1+u)^3 \big(\sqrt{ 2\gamma_{n-1}(u^2)/ \lambda}+1 \big)^2 -1\right),
    $$
    with probability at least $1-\lambda$.
\end{theorem}

\begin{proof}
    Equation~(\ref{eq:textbook-model}) states that
    $\abs{ \widehat{y} - y } \leq \abs{\sum_{i=1}^{n} x_i^2 (\psi_i -1)  } + \frac{1}{n}\mathcal{B}$. 
     The quantities $\abs{\sum_{i=1}^{n} x_i^2 (\psi_i -1) }$ and $\abs{ \widehat{s} -s }$ represent the absolute errors of the inner product $\sum_{i=1}^{n} x_i^2$ of the vector $x$ by itself and the summation $s= \sum_{i=1}^{n} x_i$, respectively. Then~\cite[sec 5.1]{arar2022stochastic} proves that
    \begin{align*}
        \abs{\sum_{i=1}^{n} x_i^2 (\psi_i -1)  } \leq \norm{x}_2^2 \sqrt{ 2\gamma_{n+1}(u^2)/ \lambda} & &\text{with probability at least $1-\frac{\lambda}{2}$},\\
        \abs{ \widehat{s} -s }\leq  \norm{x}_1 \sqrt{2 \gamma_{n-1}(u^2)/ \lambda} & & \text{with probability at least $1-\frac{\lambda}{2}$}.
    \end{align*}
    Since, $\abs{\psi_{n+1}} \leq (1+u)^3$ and $\abs{s} \leq \norm{x}_1$, with probability at least $1-\frac{ \lambda}{2},$
    \begin{align*}
        \mathcal{B} & \leq (1+u)^3 \norm{x}_1^2 \left(2\gamma_{n-1}(u^2)/ \lambda + 2\sqrt{ 2 \gamma_{n-1}(u^2)/ \lambda} \right) + \norm{x}_1^2 \left((1+u)^3 -1\right) \\
                   & = (1+u)^3 \norm{x}_1^2 \left(2\gamma_{n-1}(u^2)/ \lambda + 2\sqrt{ 2\gamma_{n-1}(u^2)/ \lambda} +1 \right) - \norm{x}_1^2                          \\
                   & = (1+u)^3 \norm{x}_1^2 \left(\sqrt{ 2\gamma_{n-1}(u^2)/ \lambda} +1 \right)^2 - \norm{x}_1^2.
    \end{align*}
    Finally, Lemma~\ref{lem:proba} shows that with probability at least $1-\lambda$,
    \begin{align*}
        \frac{\abs{ \widehat{y} - y }}{\abs{ y }} & \leq \frac{1}{\abs{y}}\abs{\sum_{i=1}^{n} x_i^2 (\psi_i -1)  } + \frac{1}{n\abs{y}} \mathcal{B}                                                          \\
                                                  & \leq \mathcal{K}_2^2 \sqrt{ 2\gamma_{n+1}(u^2)/ \lambda} + \mathcal{K}_1^2 \left( (1+u)^3 \big(\sqrt{ 2\gamma_{n-1}(u^2)/ \lambda}+1 \big)^2 -1\right).
    \end{align*}

\end{proof}

\subsubsection{Two-pass algorithm}
\label{sec-var:two-pass}
As with the previous algorithm, we present a computational scheme for the proofs of the two-pass algorithm errors in this paper. One needs first to separate the errors of order two. Let us recall that $
    \psi_i = \varphi_i (1+\epsilon_i)$  for all $1 \leq i \leq n$. Therefore
\begin{align*}
    \abs{ \widehat{z} - z } & = \abs{\sum_{i=1}^n (x_i - \widehat{m})^2 \psi_i -  (x_i - m)^2 }                                                                                             \\
                            & = \abs{\sum_{i=1}^n (x_i - \widehat{m})^2 \varphi_i - (x_i - m)^2 + \sum_{i=1}^n  (x_i - \widehat{m})^2 \epsilon_i \varphi_i }                                \\
                            & \leq \abs{\sum_{i=1}^n  (x_i - \widehat{m})^2 \varphi_i -  (x_i - m)^2 } + u\abs{\sum_{i=1}^n  (x_i - \widehat{m})^2 \varphi_i }                              \\
                            & \leq \abs{\sum_{i=1}^n  (x_i - \widehat{m})^2 \varphi_i -  (x_i - m)^2 } + u \abs{\sum_{i=1}^n  (x_i - \widehat{m})^2 \varphi_i - (x_i - m)^2 } + u \abs{ z } \\
                            &= (1+u)\abs{\sum_{i=1}^n  (x_i - \widehat{m})^2 \varphi_i -  (x_i - m)^2 } + u \abs{z}.
\end{align*}
Since $ (x_i - \widehat{m}) = (x_i -m) + (m- \widehat{m}) $, 
\begin{align*}
    \abs{\sum_{i=1}^n  (x_i - \widehat{m})^2 \varphi_i -  (x_i - m)^2}
    \leq & \abs{  \sum_{i=1}^n (x_i - m)^2 (\varphi_i-1)}  + \abs{(m - \widehat{m})^2  \sum_{i=1}^n  \varphi_i} \\
         & + 2 \abs{(m - \widehat{m} ) \sum_{i=1}^n (x_i - m)(\varphi_i-1)},
\end{align*}
because $\sum_{i=1}^n (x_i - m)=0$. Denote $$\mathcal{C} = \abs{  \sum_{i=1}^n (x_i - m)^2 (\varphi_i-1)}  
          + 2 \abs{(m - \widehat{m} ) \sum_{i=1}^n (x_i - m)(\varphi_i-1)} + \abs{(m - \widehat{m})^2  \sum_{i=1}^n  \varphi_i}.$$
The following equation will be used in all proofs of the two-pass forward errors
\begin{equation}
    \label{eq:two-pass-model}
    \abs{ \widehat{z} - z } \leq (1+u) \mathcal{C} + u\abs{z}.
\end{equation}

The following theorem presents a probabilistic bound of the forward error of this algorithm through the BC method.

\begin{theorem}
    \label{thm:var-twopass}
    For all $0 < \lambda <1$, the computed $\widehat{z}$ in Equation~(\ref{eq:z}) satisfies under SR-nearness
    \begin{align*}
        \frac{\abs{ \widehat{z} - z }}{\abs{ z }}  \leq & (1+u) \left(\sqrt{\frac{4\gamma_{n+1}(u^2)}{\lambda}} + \frac{4\gamma_{n+1}(u^2)}{\lambda} \left( 2 \mathcal{K}_1 + \mathcal{K}_1^2  \left(\sqrt{ \frac{4\gamma_{n+1}(u^2)}{\lambda}} + 1\right)
            \right)  \right)  \\
            &+u,
    \end{align*}
    with probability at least $1-\lambda$.
\end{theorem}

\begin{proof}
    Equation~(\ref{eq:two-pass-model}) states that
    $\abs{ \widehat{z} - z } \leq (1+u) \mathcal{C} + u\abs{z}$,
    and $\abs{\sum_{i=1}^n  \varphi_i} \leq \abs{\sum_{i=1}^n  (\varphi_i -1)} + n$. The following~ quantities $\abs{\sum_{i=1}^n (x_i - m)^2 (\varphi_i-1)}$, $\abs{ \widehat{m}  - m }$,  $\abs{\sum_{i=1}^n (x_i - m) (\varphi_i-1)} $ and $\abs{\sum_{i=1}^n  (\varphi_i-1)} $ represent the absolute errors of the inner product of $x-m$ by itself $\sum_{i=1}^{n} (x_i -m)^2$, the average $m=\frac{1}{n} \sum_{i=1}^{n} x_i $, the summations $s= \sum_{i=1}^{n} (x_i -m)$ and $\sum_{i=1}^{n} 1$ respectively. Then~\cite[sec 5.1]{arar2022stochastic} proves that
    \begin{align*}
        \abs{\sum_{i=1}^n (x_i - m)^2 (\varphi_i-1)} & \leq \abs{ z }\sqrt{\frac{4\gamma_{n+1}(u^2)}{\lambda}}                   &  & \text{with probability at least $1-\frac{\lambda}{4}$}, \\
        \abs{ \widehat{m}  - m }                     & \leq \frac{1}{n}\norm{x}_1 \sqrt{\frac{4\gamma_{n}(u^2)}{\lambda}}        &  & \text{with probability at least $1-\frac{\lambda}{4}$}, \\
        \abs{\sum_{i=1}^n (x_i - m) (\varphi_i-1)}   & \leq \sum_{i=1}^n \abs{ x_i - m}\sqrt{\frac{4\gamma_{n+1}(u^2)}{\lambda}} &  & \text{with probability at least $1-\frac{\lambda}{4}$}, \\
        \abs{\sum_{i=1}^n  (\varphi_i-1)}            & \leq n \sqrt{\frac{4\gamma_{n+1}(u^2)}{\lambda}} +n                       &  & \text{with probability at least $1-\frac{\lambda}{4}$}.
    \end{align*}
    Using the Cauchy–Schwarz inequality, we obtain $$ \sum_{i=1}^n \abs{ x_i - m}\leq \sqrt{n\sum_{i=1}^n (x_i - m)^2} = \sqrt{n z}.$$
    Since $\gamma_n(u^2) \leq \gamma_{n+1}(u^2)$, Lemma~\ref{lem:proba} implies 
    \begin{align*}
        \mathcal{C} & \leq \abs{ z }\sqrt{ \frac{4\gamma_{n+1}(u^2)}{\lambda}} + 2 \frac{\norm{x}_1 }{n}\frac{4\gamma_{n+1}(u^2)}{\lambda}\sqrt{nz}
        + \frac{\norm{x}_1^2 }{n}\frac{4\gamma_{n+1}(u^2)}{\lambda} \left( \sqrt{ \frac{4\gamma_{n+1}(u^2)}{\lambda}} + 1\right)                           \\
            &=       \abs{ z }\sqrt{ \frac{4\gamma_{n+1}(u^2)}{\lambda}} +  \frac{4\gamma_{n+1}(u^2)}{\lambda} \left( 2 \abs{ z }\frac{\norm{x}_1}{\sqrt{nz}}
        + \frac{\norm{x}_1^2}{n} \left(\sqrt{ \frac{4\gamma_{n+1}(u^2)}{\lambda}} + 1\right)\right),
    \end{align*}
    with probability at least $1-\lambda$. Finally
    \begin{align*}
        \frac{\abs{ \widehat{z} - z }}{\abs{ z }}  \leq & (1+u) \left(\sqrt{\frac{4\gamma_{n+1}(u^2)}{\lambda}} + \frac{4\gamma_{n+1}(u^2)}{\lambda} \left( 2 \mathcal{K}_1 + \mathcal{K}_1^2  \left(\sqrt{ \frac{4\gamma_{n+1}(u^2)}{\lambda}} + 1\right)
            \right)  \right)  \\
            &+u,
    \end{align*}
    with probability at least $1-\lambda$,
\end{proof}

\subsection{AH method}
\label{sec:martingale-method}
This section uses the AH method proposed in~\cite{ilse} for the inner product and Lemma~\ref{lem:proba} to provide a probabilistic bound of the forward error of both textbook and two-pass algorithms under SR-nearness.

\subsubsection{Textbook algorithm}
\label{sec-mar:textbook}

\begin{theorem}
    \label{thm:mar-textbook}
    For all $0 < \lambda <1$, the computed $\widehat{y}$ in Equation~(\ref{eq:y}) satisfies under SR-nearness
    \begin{align*}
        \frac{\abs{ \widehat{y}  - y}}{\abs{ y}} \leq & \mathcal{K}_2^2 \sqrt{ u \gamma_{2(n+1)}(u) } \sqrt{\ln (4 / \lambda)}                                             \\
                                                      & + \mathcal{K}_1^2 \left( (1+u)^3 \big(\sqrt{ u \gamma_{2(n-1)}(u) } \sqrt{\ln (4 / \lambda)} +1 \big)^2 -1\right),
    \end{align*}
    with probability at least $1-\lambda$.
\end{theorem}

\begin{proof}
    Equation~(\ref{eq:textbook-model}) states that
    $\abs{ \widehat{y} - y } \leq \abs{\sum_{i=1}^{n} x_i^2 (\psi_i -1)  } + \frac{1}{n}\mathcal{B}.$
    Moreover,~\cite[cor 4.7]{ilse} shows that
    \begin{align*}
        \abs{  \sum_{i=1}^{n} x_i^2 (\psi_i -1) }  \leq \norm{x}_2^2 \sqrt{ u\gamma_{2(n+1)}(u) } \sqrt{\ln (4 / \lambda)} & & \text{with probability at least $1-\frac{\lambda}{2}$}, \\
        \abs{ \widehat{s} -s }                     \leq  \norm{x}_1 \sqrt{ u\gamma_{2(n-1)}(u) } \sqrt{\ln (4 / \lambda)} & & \text{with probability at least $1-\frac{\lambda}{2}$}.
    \end{align*}
    Since, $\abs{\psi_{n+1}} \leq (1+u)^3$ and $\abs{s} \leq \norm{x}_1$,we have with probability at least $1-\frac{ \lambda}{2}$,
    \begin{align*}
        \scriptsize \mathcal{B} \leq & (1+u)^3 \norm{x}_1^2 \left( u\gamma_{2(n-1)}(u) \ln (4 / \lambda) + 2\sqrt{ u\gamma_{2(n-1)}(u) } \sqrt{\ln (4 / \lambda)} \right)                 \\
                                    & + \norm{x}_1^2 \left((1+u)^3 -1\right)                                                                                                             \\
        =                           & (1+u)^3 \norm{x}_1^2 \left(u\gamma_{2(n-1)}(u) \ln (4 / \lambda)+ 2\sqrt{ u\gamma_{2(n-1)}(u) } \sqrt{\ln (4 / \lambda)} +1 \right) - \norm{x}_1^2 \\
        =                           & (1+u)^3 \norm{x}_1^2 \left(\sqrt{ u\gamma_{2(n-1)}(u) } \sqrt{\ln (4 / \lambda)} +1 \right)^2 - \norm{x}_1^2.
    \end{align*}
     Finally, Lemma~\ref{lem:proba} shows that with probability at least $1-\lambda$,
    \begin{align*}
        \frac{\abs{ \widehat{y} - y }}{\abs{ y }} \leq & \mathcal{K}_2^2 \sqrt{ u\gamma_{2(n+1)}(u) } \sqrt{\ln (4 / \lambda)}                                                 \\
                                                       & + \mathcal{K}_1^2 \left( (1+u)^3 \left(\sqrt{ u\gamma_{2(n-1)}(u) } \sqrt{\ln (4 / \lambda)} + 1 \right)^2 -1\right).
    \end{align*}
    
\end{proof}

\subsubsection{Two-pass algorithm}
\label{sec-mar:two-pass}

\begin{theorem}
    For all $0 < \lambda <1$, the computed $\widehat{z}$ in Equation~(\ref{eq:z}) satisfies under SR-nearness
    \begin{align*}
        \frac{\abs{ \widehat{z} - z }}{\abs{ z }} \leq & (1+u) \Bigg( \sqrt{ u\gamma_{2(n+1)}(u) } \sqrt{\ln (8 / \lambda)} \\
                                                       & + u\gamma_{2(n+1)}(u)\ln (8 / \lambda) \left( 2 \mathcal{K}_1
            + \mathcal{K}_1^2   \big( \sqrt{ u\gamma_{2(n+1)}(u) } \sqrt{\ln (8 / \lambda)} + 1\big)
            \right) \Bigg)  +u,
    \end{align*}
    with probability at least $1-\lambda$.

\end{theorem}

\begin{proof}
    Equation~(\ref{eq:two-pass-model}) states that
    $\abs{ \widehat{z} - z } \leq (1+u) \mathcal{C} + u\abs{z}.$
    Note that $\abs{\sum_{i=1}^n  \varphi_i} \leq \abs{\sum_{i=1}^n  (\varphi_i -1)} + n$ and~\cite[cor 4.7]{ilse} shows that each of the following inequalities holds with probability at least $1-\frac{\lambda}{4}$:
    \begin{align*}
        \abs{\sum_{i=1}^n (x_i - m)^2 (\varphi_i-1)} & \leq \abs{ z } \sqrt{ u \gamma_{2(n+1)}(u) } \sqrt{\ln (8 / \lambda)},                  \\
        \abs{ \widehat{m}  - m }                     & \leq \frac{1}{n}\norm{x}_1  \sqrt{ u \gamma_{2n}(u) } \sqrt{\ln (8 / \lambda)},         \\
        \abs{\sum_{i=1}^n (x_i - m) (\varphi_i-1)}   & \leq \sum_{i=1}^n \abs{ x_i - m}\sqrt{ u \gamma_{2(n+1)}(u) } \sqrt{\ln (8 / \lambda)}, \\
        \abs{\sum_{i=1}^n  (\varphi_i -1)}           & \leq n  \sqrt{ u \gamma_{2(n+1)}(u) } \sqrt{\ln (8 / \lambda)}.       
    \end{align*}
    By the Cauchy–Schwarz inequality, 
    $
        \sum_{i=1}^n \abs{ x_i - m}\leq \sqrt{n\sum_{i=1}^n (x_i - m)^2} = \sqrt{nz}
    $. Since $\gamma_{2n}(u) \leq \gamma_{2(n+1)}(u)$, Lemma~\ref{lem:proba} implies 
    \begin{align*}
        \mathcal{C} \leq & \abs{ z } \sqrt{ u \gamma_{2(n+1)}(u)} \sqrt{ \ln (8 / \lambda)} + 2 \frac{\norm{x}_1}{n} u\gamma_{2(n+1)}(u)   \ln (8 / \lambda) \sqrt{nz}          \\
                        & + \frac{\norm{x}_1^2}{n^2} u\gamma_{2(n+1)}(u)  \ln (8 / \lambda) \left(n \sqrt{ u\gamma_{2(n+1)}(u) } \sqrt{ \ln (8 / \lambda)} + n\right)           \\
        =               & \abs{ z } \sqrt{ u\gamma_{2(n+1)}(u) } \sqrt{\ln (8 / \lambda)} + u\gamma_{2(n+1)}(u) \ln (8 / \lambda)\Bigg( 2 \abs{ z }\frac{\norm{x}_1}{\sqrt{nz}} \\
                        & + \frac{\norm{x}_1^2}{n} \Big( \sqrt{ u\gamma_{2(n+1)}(u) } \sqrt{ \ln (8 / \lambda)} + 1\Big)\Bigg),
    \end{align*}
    with probability at least $1-\lambda$, Finally
    \begin{align*}
        \frac{\abs{ \widehat{z} - z }}{\abs{ z }} \leq & (1+u) \Bigg( \sqrt{ u\gamma_{2(n+1)}(u) } \sqrt{\ln (8 / \lambda)} \\
                                                       & + u\gamma_{2(n+1)}(u)\ln (8 / \lambda) \left( 2 \mathcal{K}_1
            + \mathcal{K}_1^2   \big( \sqrt{ u\gamma_{2(n+1)}(u) } \sqrt{\ln (8 / \lambda)} + 1\big)
            \right) \Bigg)  +u,
    \end{align*}
    with probability at least $1-\lambda$.
\end{proof}

\subsubsection{Textbook algorithm and Doob-Meyer decomposition}
\label{sec-doob:textbook}
This work introduces a new approach based on Doob–Meyer decomposition~\cite[p 68]{dacunha2012probability} to bound the forward error of the textbook algorithm.
To apply this method, we study
$$\widehat{s}= \sum_{i=1}^n x_i  \prod_{k=\max(2, i)}^{n}(1+\delta_{k-1}).$$
Consider $s_1 = x_1, \ s_{k} = s_{k-1} +x_{k}$ and $\widehat s_1 = x_1, \ \widehat{s}_{k} = (\widehat s_{k-1} +x_{k})(1+\delta_{k-1}) $ for all $2 \leq k \leq n$. Then $s_{n}= s $ and $\widehat{s}_{n}= \widehat{s}$.
Denote $Z_k = \widehat{s}_{k} - s_k = Z_{k-1} + (\widehat{s}_{k-1} +x_k)\delta_{k-1}$. Then, $Z_n= \widehat{s}_{n} - s_n$. By mean independence of $ \delta_k$,  $Z_1,\ldots,Z_n$ form a martingale with respect to $\delta_1,\ldots,\delta_{n-1}$. Then, $Z_1 +s,\ldots,Z_n +s$ is also a martingale.
Denote:
\begin{itemize}
    \item $\mathbb{F}_k= \{\delta_1, \ldots, \delta_k\}$.
    \item $Y_{k-1} = Z_k - Z_{k-1} = (\widehat{s}_{k-1} +x_k)\delta_{k-1}$ for all $2\leq  k \leq n$. Then $Z_n = \sum_{k=2}^n Y_{k-1}$.
    \item $\sigma_{k-1}^2 = E[Y_{k-1}^2/\mathbb{F}_{k-2}]$.
    \item $A_n = \sum_{k=2}^n \sigma_{k-1}^2$ with $A_1 =0$.
\end{itemize}
On one hand, $A_n$ is predictable:
\begin{align*}
		E[A_n/\mathbb{F}_{n-1}] & = E\left[\sum_{k=2}^n \sigma_{k-1}^2/\mathbb{F}_{n-1}\right]\\
		& = E\left[\sum_{k=2}^n E\left[Y_{k-1}^2/\mathbb{F}_{k-2}\right]/\mathbb{F}_{n-1}\right]            \\
		& = \sum_{k=2}^n E[ E[Y_{k-1}^2/\mathbb{F}_{k-2}]/\mathbb{F}_{n-1}].
	\end{align*}
	Since $E[Y_{k-1}^2/\mathbb{F}_{k-2}]$ is $\mathbb{F}_{k-2}$-measurable, so it is $\mathbb{F}_{n-1}$-measurable, and for all $2 \leq k \leq n$, we have $ E[ E[Y_{k-1}^2/\mathbb{F}_{k-2}]/\mathbb{F}_{n-1}] =  E[Y_{k-1}^2/\mathbb{F}_{k-2}]$ . Then
    \begin{align*}
		E[A_n/\mathbb{F}_{n-1}]	= \sum_{k=2}^n E[Y_{k-1}^2/\mathbb{F}_{k-2}]   
		= A_n.
    \end{align*}
On the other hand, $X_n = (Z_n +s)^2 - A_n - s^2$ is a martingale:
\begin{align*}
    E[X_n/\mathbb{F}_{n-1}] &=  E[(Z_n +s)^2 - A_n -s^2/\mathbb{F}_{n-1}] \\
    &=  E[(Z_{n-1} +s + Y_{n-1})^2 /\mathbb{F}_{n-1}] - A_n -s^2\\
    &= (Z_{n-1}+s)^2 + 2 (Z_{n-1}+s) E[Y_{n-1}/\mathbb{F}_{n-1}] +  E[Y_{n-1}^2/\mathbb{F}_{n-1}] - A_n -s^2\\
    &= X_{n-1} \quad \text{because} \ E[Y_{n-1}/\mathbb{F}_{n-1}]  =0.
\end{align*}
The expression of $(Z_n +s)^2 =X_n + s^2 + A_n $ is a Doob-Meyer decomposition.

\begin{lemma}\label{cst-bound}
    The martingale $X_1,\ldots, X_{n}$ satisfies
    $\abs{X_k - X_{k-1}} \leq u C_k $, for all $2\leq k \leq n$,
    where
    $$ C_k = \norm{x}_1^2 (1+u)^{2(k-2)} (2+u). $$
\end{lemma}

\begin{proof}
     Note that $\sigma_{k-1}^2 = E[(\widehat{s}_{k-1} +x_k)^2\delta_{k-1}^2/\mathbb{F}_{k-2}] = (\widehat{s}_{k-1} +x_k)^2E[\delta_{k-1}^2/\mathbb{F}_{k-2}]$ by definition of $\mathbb{F}_{k-2}$. Then
    \begin{align*}
        X_k - X_{k-1} &= (Z_k +s)^2 - A_k  -(Z_{k-1} +s)^2 + A_{k-1} \\
        &= (Z_{k-1} +s + (\widehat{s}_{k-1} +x_k)\delta_{k-1})^2 - A_k  -(Z_{k-1} +s)^2 + A_{k-1}  \\
                      & =  2(Z_{k-1} + s) (\widehat{s}_{k-1} +x_k)\delta_{k-1} + (\widehat{s}_{k-1} +x_k)^2\delta_{k-1}^2 -\sigma_{k-1}^2                         \\
                      & = 2(Z_{k-1}+s) (\widehat{s}_{k-1} +x_k)\delta_{k-1} + (\widehat{s}_{k-1} +x_k)^2 \left( \delta_{k-1}^2 -E[\delta_{k-1}^2/\mathbb{F}_{k-2}]\right).
    \end{align*}
    Since $\abs{\delta_{k-1}} \leq u$, we have  $\abs{\widehat{s}_{k-1} +x_k} \leq (1+u)^{k-2} \sum_{i=1}^{k} \abs{x_i} \leq (1+u)^{k-2}\norm{x}_1 $, $\abs{\delta_{k-1}^2 -E[\delta_{k-1}^2/\mathbb{F}_{k-2}]} \leq u^2$ because $0 \leq \delta_{k-1}^2 \leq u^2$ and 
    $$  \abs{Z_{k-1} +s} \leq ((1+u)^{k-2} -1)  \sum_{i=1}^{k-1} \abs{x_i} + \abs{s} \leq \norm{x}_1 (1+u)^{k-2}. $$
    Thus
    \begin{align*}
        \abs{X_k - X_{k-1} } & \leq 2 u \abs{Z_{k-1}+s}  \abs{\widehat{s}_{k-1} +x_k} + u^2 \abs{\widehat{s}_{k-1} +x_k}^2               \\
                             & \leq 2u (1+u)^{2(k-2)} \norm{x}_1^2  + u^2(1+u)^{2(k-2)} \norm{x}_1^2  \\
                             & = u \norm{x}_1^2 (1+u)^{2(k-2)} (2+u).
    \end{align*}
\end{proof}

\begin{theorem}
\label{thm:x_n}
    For $0 < \lambda <1$, the martingale $X_1,\ldots, X_{n}$ satisfies under SR-nearness
    \begin{equation}
        \label{eq:x_n}
        \abs{X_n} \leq  \norm{x}_1^2   \sqrt{2u\gamma_{4(n-1)}(u)}  \sqrt{\ln(2/ \lambda)},
    \end{equation}
    with probability at least $1-\lambda$.
\end{theorem}

\begin{proof}
    Since $X_1=0$, Lemma~\ref{hoeffding} and Lemma~\ref{cst-bound} yields
    $$
        \abs{X_n} \leq  \sqrt{\sum_{k=2}^n u^2 C_k^2} \sqrt{2\ln(2/ \lambda)},$$
    with probability at least $1-\lambda$. Furthermore
    \begin{align*}
        \sum_{k=2}^n u^2 C_k^2 & = u^2 \sum_{k=2}^n  \norm{x}_1^4 (1+u)^{4(k-2)} (2+u)^2                                                        = u^2\norm{x}_1^4 (2+u)^2 \frac{\gamma_{4(n-1)}(u)}{(1+u)^4 -1}                                                 \\
                           & = u \norm{x}_1^4  \frac{4 + 4u + u^2}{4 + 6u + 4u^2 + u^3} \gamma_{4(n-1)}(u)                                                                                         \\
                           & \leq u \norm{x}_1^4 \gamma_{4(n-1)}(u).
    \end{align*}
    Finally, $ \abs{X_n} \leq  \norm{x}_1^2  \sqrt{2u \gamma_{4(n-1)}(u) } \sqrt{\ln(2/ \lambda)}. $    
\end{proof}
We are now in a position to state the main result of this sub-section.

\begin{theorem}
    For all $0 < \lambda <1$, the computed $\widehat{y}$ in Equation~(\ref{eq:y}) satisfies under SR-nearness
    \begin{align*}
        \frac{\abs{\widehat{y} - y}}{\abs{y}} \leq & \mathcal{K}_2^2 \sqrt{ u \gamma_{2(n+1)}(u) } \sqrt{\ln (4 / \lambda)} +
        \mathcal{K}_1^2 (1+u)^3 \big[ \sqrt{2u\gamma_{4(n-1)}(u)} \sqrt{\ln(4/ \lambda)}   \\            & + u\frac{\gamma_{2(n-1)}(u)}{2} +1\big] - \mathcal{K}_1^2,
    \end{align*}
    with probability at least $1-\lambda$. In the following, this bound will be called DM bound.
\end{theorem}

\begin{proof}
    Recall that $Z_n= \widehat{s} - s$ and $ (Z_n+s)^2 = X_n +s^2  + A_n$. Therefore, from Sub-section~\ref{sec:bias-textbook},
    \begin{align*}
        \widehat{y} - y & =  \sum_{i=1}^{n} x_i^2 (\psi_i -1) - \frac{1}{n} \widehat{s}^2\psi_{n+1} + \frac{1}{n}  s^2
        =  \sum_{i=1}^{n} x_i^2 (\psi_i -1) - \frac{1}{n} (Z_n +s)^2\psi_{n+1} + \frac{1}{n}  s^2                                                 \\
                        & =   \sum_{i=1}^{n} x_i^2 (\psi_i -1) - \frac{1}{n}\psi_{n+1}(X_n  + A_n)  - \frac{1}{n}  s^2(\psi_{n+1} -1).
    \end{align*}
    Since $\abs{\psi_{n+1}} \leq (1+u)^3$ and $\abs{s} \leq \norm{x}_1$, we deduce that
    \begin{align*}
        \abs{\widehat{y} - y} & \leq \abs{  \sum_{i=1}^{n} x_i^2 (\psi_i -1) } + \frac{1}{n} (1+u)^3 \left( \abs{X_n} + \abs{A_n} \right) + \frac{1}{n} \norm{x}_1^2 \gamma_3(u)    \\
                              & = \abs{  \sum_{i=1}^{n} x_i^2 (\psi_i -1) } + \frac{1}{n} (1+u)^3 \left( \abs{X_n} + \abs{A_n} + \norm{x}_1^2  \right) - \frac{1}{n} \norm{x}_1^2.
    \end{align*}
    On one hand, Theorem~\ref{thm:x_n} states that with probability at least $1-\frac{\lambda}{2}$,
    $$ \abs{X_n} \leq \norm{x}_1^2   \sqrt{2u\gamma_{4(n-1)}(u)} \sqrt{\ln(4/ \lambda)}.
    $$
    On the other hand, $A_n =  \sum_{k=2}^n E[Y_{k-1}^2/\mathbb{F}_{k-2}] = \sum_{k=2}^n   (\widehat{s}_{k-1} +x_k)^2E[\delta_{k-1}^2/\mathbb{F}_{k-2}]$, then
    \begin{align*}
        \abs{A_n} & \leq u^2\sum_{k=2}^n \abs{\widehat{s}_{k-1} +x_k}^2
        \leq u^2 \sum_{k=2}^n \left( (1+u)^{k-2} \sum_{i=1}^{k} \abs{x_i} \right)^2 \\
                  & \leq u^2 \norm{x}_1^2 \sum_{k=2}^n (1+u)^{2(k-2)}
        \leq u^2 \norm{x}_1^2 \frac{\gamma_{2(n-1)}(u)}{2u + u^2} \leq u \norm{x}_1^2 \frac{\gamma_{2(n-1)}(u)}{2}.
    \end{align*}
    Moreover~\cite[cor 4.7]{ilse} yields:
    \begin{align*}
        \abs{  \sum_{i=1}^{n} x_i^2 (\psi_i -1) } & \leq \norm{x}_2^2 \sqrt{ u\gamma_{2(n+1)}(u) } \sqrt{\ln (4 / \lambda)} &  & \text{with probability at least $1- \frac{\lambda}{2}$}.
        \end{align*}
    Finally, Lemma~\ref{lem:proba} implies 
    \begin{align*}
        \frac{\abs{\widehat{y} - y}}{\abs{y}} \leq & \frac{\norm{x}_2^2}{\abs{y}}  \sqrt{ u \gamma_{2(n+1)}(u) } \sqrt{\ln (4 / \lambda)} + \frac{\norm{x}_1^2}{n\abs{y}} (1+u)^3 \Bigg(  \sqrt{2u\gamma_{4(n-1)}(u)} \sqrt{\ln(4/ \lambda)} \\
                                                   & + u\frac{\gamma_{2(n-1)}(u)}{2} +1 \Bigg) - \frac{\norm{x}_1^2}{n\abs{y}}          \\
        =                                          & \mathcal{K}_2^2 \sqrt{ u \gamma_{2(n+1)}(u) } \sqrt{\ln (4 / \lambda)} + \mathcal{K}_1^2 (1+u)^3 \Bigg( \sqrt{2u\gamma_{4(n-1)}(u)} \sqrt{\ln(4/ \lambda)}   \\            & + u\frac{\gamma_{2(n-1)}(u)}{2} +1\Bigg) - \mathcal{K}_1^2,
    \end{align*}
     with probability at least $1-\lambda$.
\end{proof}

\section{Pairwise textbook and pairwise two-pass}
\label{sec:parwise-non-linear}
In this section, we illustrate the continued applicability of SR results on the forward error of the pairwise summation to the forward error of both pairwise textbook and pairwise two-pass algorithms (ie. the two-pass and textbook algorithms in which sums are computed pairwise). The following theorem derives a probabilistic bound for the pairwise textbook using the BC method.

\begin{theorem}
    For the pairwise textbook algorithm, for all $0 < \lambda <1$, the computed $\widehat{y}$ in Equation~(\ref{eq:y}) satisfies under SR-nearness
    $$  \frac{\abs{ \widehat{y} - y }}{\abs{ y }}
        \leq  \mathcal{K}_2^2 \sqrt{ 2\gamma_{\log(n)+1}(u^2)/ \lambda} + \mathcal{K}_1^2 \left( (1+u)^3 \big(\sqrt{ 2\gamma_{\log(n)}(u^2)/ \lambda}+1 \big)^2 -1\right),
    $$
    with probability at least $1-\lambda$.
\end{theorem}

\begin{proof}
    Equation~(\ref{eq:textbook-model}) states that
    $\abs{ \widehat{y} - y } \leq \abs{\sum_{i=1}^{n} x_i^2 (\psi_i -1)  } + \frac{1}{n}\mathcal{B}$. Since the sum is pairwise, the term $\prod_{k=\max\{2, i\}}^{n+1}(1+\eta_k)$ in $\psi_i$ can be replaced with a term $\prod_{k=1}^{\log(n)}(1+\eta_k)$ as shown in
    Section~\ref{sec:parwise}. Thus:
    \begin{align*}
        \abs{\sum_{i=1}^{n} x_i^2 (\psi_i -1)  } \leq \norm{x}_2^2 \sqrt{ 2\gamma_{\log(n)+1}(u^2)/ \lambda} & & \text{with probability at least $1-\frac{\lambda}{2}$},\\
        \abs{ \widehat{s} -s }\leq  \norm{x}_1 \sqrt{2 \gamma_{\log(n)}(u^2)/ \lambda} & & \text{with probability at least $1-\frac{\lambda}{2}$}.
    \end{align*}
    Since, $\abs{\psi_{n+1}} \leq (1+u)^3$ and $\abs{s} \leq \norm{x}_1$, we have with probability at least $1-\frac{ \lambda}{2}$
    \begin{align*}
        \mathcal{B} & \leq (1+u)^3 \norm{x}_1^2 \left(2\gamma_{\log(n)}(u^2)/ \lambda + 2\sqrt{ 2\gamma_{\log(n)}(u^2)/ \lambda} \right) + \norm{x}_1^2 \left((1+u)^3 -1\right) \\
                   & = (1+u)^3 \norm{x}_1^2 \left(2\gamma_{\log(n)}(u^2)/ \lambda + 2\sqrt{ 2\gamma_{\log(n)}(u^2)/ \lambda} +1 \right) - \norm{x}_1^2                          \\
                   & = (1+u)^3 \norm{x}_1^2 \left(\sqrt{ 2\gamma_{\log(n)}(u^2)/ \lambda} +1 \right)^2 - \norm{x}_1^2.
    \end{align*}
    Finally, Lemma~\ref{lem:proba} shows that with probability at least $1-\lambda$,
    \begin{align*}
        \frac{\abs{ \widehat{y} - y }}{\abs{ y }} & \leq \frac{1}{\abs{y}}\abs{\sum_{i=1}^{n} x_i^2 (\psi_i -1)  } + \frac{1}{n\abs{y}} \mathcal{B}                                                                    \\
                                                  & \leq \mathcal{K}_2^2 \sqrt{ 2\gamma_{\log(n)+1}(u^2)/ \lambda} + \mathcal{K}_1^2 \left( (1+u)^3 \big(\sqrt{ 2\gamma_{\log(n)}(u^2)/ \lambda}+1 \big)^2 -1\right).
    \end{align*}
\end{proof}

The following theorem shows the probabilistic bound for the pairwise textbook algorithm using the AH method.

\begin{theorem}
    For the pairwise textbook algorithm, for all $0 < \lambda <1$, the computed $\widehat{y}$ in Equation~(\ref{eq:y}) satisfies under SR-nearness
    \begin{align*}
        \frac{\abs{ \widehat{y}  - y}}{\abs{ y}} \leq & \mathcal{K}_2^2 \sqrt{ u \gamma_{2(\log(n)+1)}(u) } \sqrt{\ln (4 / \lambda)} \\ &+ \mathcal{K}_1^2 \left( (1+u)^3 \big(\sqrt{ u \gamma_{2\log(n)}(u) } \sqrt{\ln (4 / \lambda)} +1 \big)^2 -1\right),
    \end{align*}
    with probability at least $1-\lambda$.
\end{theorem}

\begin{proof}
    Equation~(\ref{eq:textbook-model}) states that
    $\abs{ \widehat{y} - y } \leq \abs{\sum_{i=1}^{n} x_i^2 (\psi_i -1)  } + \frac{1}{n}\mathcal{B}.$ 
    Moreover, Section~\ref{sec:parwise} shows
    \begin{align*}
        \abs{  \sum_{i=1}^{n} x_i^2 (\psi_i -1) }  \leq \norm{x}_2^2 \sqrt{ u\gamma_{2(\log(n)+1)}(u) } \sqrt{\ln (4 / \lambda)} &  \text{ with probability at least $1-\frac{\lambda}{2}$}, \\
        \abs{ \widehat{s} -s }                     \leq  \norm{x}_1 \sqrt{ u\gamma_{2\log(n)}(u) } \sqrt{\ln (4 / \lambda)} &  \text{ with probability at least $1-\frac{\lambda}{2}$}.
    \end{align*}
    As the previous proof, we can show that with probability at least $1-\frac{\lambda}{2}$,
    $$ \mathcal{B} \leq  (1+u)^3 \norm{x}_1^2 \left(\sqrt{ u\gamma_{2\log(n)}(u) } \sqrt{\ln (4 / \lambda)} +1 \right)^2 - \norm{x}_1^2. 
                   $$
    Finally, with probability at least $1-\lambda$,
    \begin{align*}
        \frac{\abs{ \widehat{y} - y }}{\abs{ y }} \leq & \mathcal{K}_2^2 \sqrt{ u \gamma_{2(\log(n)+1)}(u) } \sqrt{\ln (4 / \lambda)}                                         \\
                                                       & + \mathcal{K}_1^2 \left( (1+u)^3 \big(\sqrt{ u \gamma_{2\log(n)}(u) } \sqrt{\ln (4 / \lambda)} +1 \big)^2 -1\right).
    \end{align*}
\end{proof}

Similar bounds are reached for the pairwise two-pass using the same methods.
\section{Error bound analysis}
\label{sec:error}
Table~\ref{table-n} shows the asymptotic forward error bounds for the textbook algorithm. Higher order terms in $u$ have been dropped when $nu \ll 1$ and uniquely for BC when $nu \gg 1$ 
and $nu^2 \ll 1$, and only dominant terms are shown. The results in the table are based on: $\gamma_n(u) \approx nu + O(n u^2)$ and $ \sqrt{u\gamma_n(u)} \approx \sqrt{\gamma_n(u^2)} \approx \sqrt{n}u + O(nu^2)$ when $nu \ll 1$. $\gamma_n(u) \approx e^{nu} $, $ \sqrt{u\gamma_n(u)} \approx \sqrt{u}e^{\frac{n}{2}u}$ and $ \sqrt{\gamma_n(u^2)} \approx \sqrt{n}u + O(nu^2)$ when $nu \gg 1$ and $nu^2 \ll 1$.

\begin{table}[h]
    {\renewcommand{\arraystretch}{1.7}

        \centering

        \begin{tabular}{|C{0.9cm}|C{4.1cm}|C{6.7cm}|}
            \cline{2-3}
            \cline{2-3}
            \multicolumn{1}{c|}{} & $nu \ll 1$               & $nu \gg 1$ and $nu^2 \ll 1$  \\
            \hline
            Det                   & $(\mathcal{K}_2^2 + 2\mathcal{K}_1^2)nu$                               & $(\mathcal{K}_2^2 + \mathcal{K}_1^2) e^{(2n+1)u}$                         \\
            \hline
            BC                    & $(\mathcal{K}_2^2 + 2\mathcal{K}_1^2 ) \sqrt{2/\lambda}\sqrt{n}u$      & $(\mathcal{K}_2^2 + 2\mathcal{K}_1^2 ) \sqrt{2/\lambda}\sqrt{n}u$       \\
            \hline
            AH                    & $(\mathcal{K}_2^2 + 2\mathcal{K}_1^2 ) \sqrt{2\ln(4/\lambda)}\sqrt{n}u$ & $(\mathcal{K}_2^2  + \mathcal{K}_1^2 \sqrt{u\ln(4/\lambda)} )  \sqrt{u\ln(4/\lambda)} e^{(2n+1)u}$ \\
            \hline
            DM                  & $(\mathcal{K}_2^2 + 2\mathcal{K}_1^2 ) \sqrt{2\ln(4/\lambda)}\sqrt{n}u$ & $\left(\sqrt{u\ln(4/\lambda)}(\mathcal{K}_2^2  + \sqrt{2}\mathcal{K}_1^2) +  \mathcal{K}_1^2\frac{u}{2} \right) e^{(2n+1)u}$ \\
            \hline
        \end{tabular}
        \caption{The asymptotic behavior of the textbook forward error bounds for a fixed probability $\lambda$ and over $n$ up to a constant.}
        \label{table-n}
    }
\end{table}

This table displays the advantage of the probabilistic bounds of the textbook forward error in terms of $O(\sqrt{n}u)$ compared to the deterministic bounds in $O(nu)$, when $nu \ll 1$. Additionally, the BC method is far better when $nu\gg1$ and $nu^2 \ll 1$. The previous discussion also holds for the two-pass forward error bounds. 

\subsection{Numerical experiments}
\label{sec:numerical-experiments}

We performed a series of numerical experiments comparing these new probabilistic bounds to the deterministic ones. We show that probabilistic bounds are tighter and accurately reflect the behavior of SR-nearness forward errors. Two types of plots are presented. Firstly, the plots are displayed over $n$ and show that for large values of $n$, BC bounds provide significant benefits compared to AH or DM bounds for the textbook algorithm.
Secondly, the plots are shown over $\lambda$, and show that AH bound holds a significant advantage for higher probabilities.
All SR computations are repeated 30 times with verificarlo~\cite{verificarlo}. All samples and the forward error of the average of the 30 SR instances are plotted.

\subsubsection{Textbook algorithm}
We present a numerical application of the textbook algorithm for floating-points chosen uniformly at random between $0$ and $1$. 
\begin{figure}
    \centering
    \subfloat{
        \hspace{-.6cm}\includegraphics[scale=0.29]{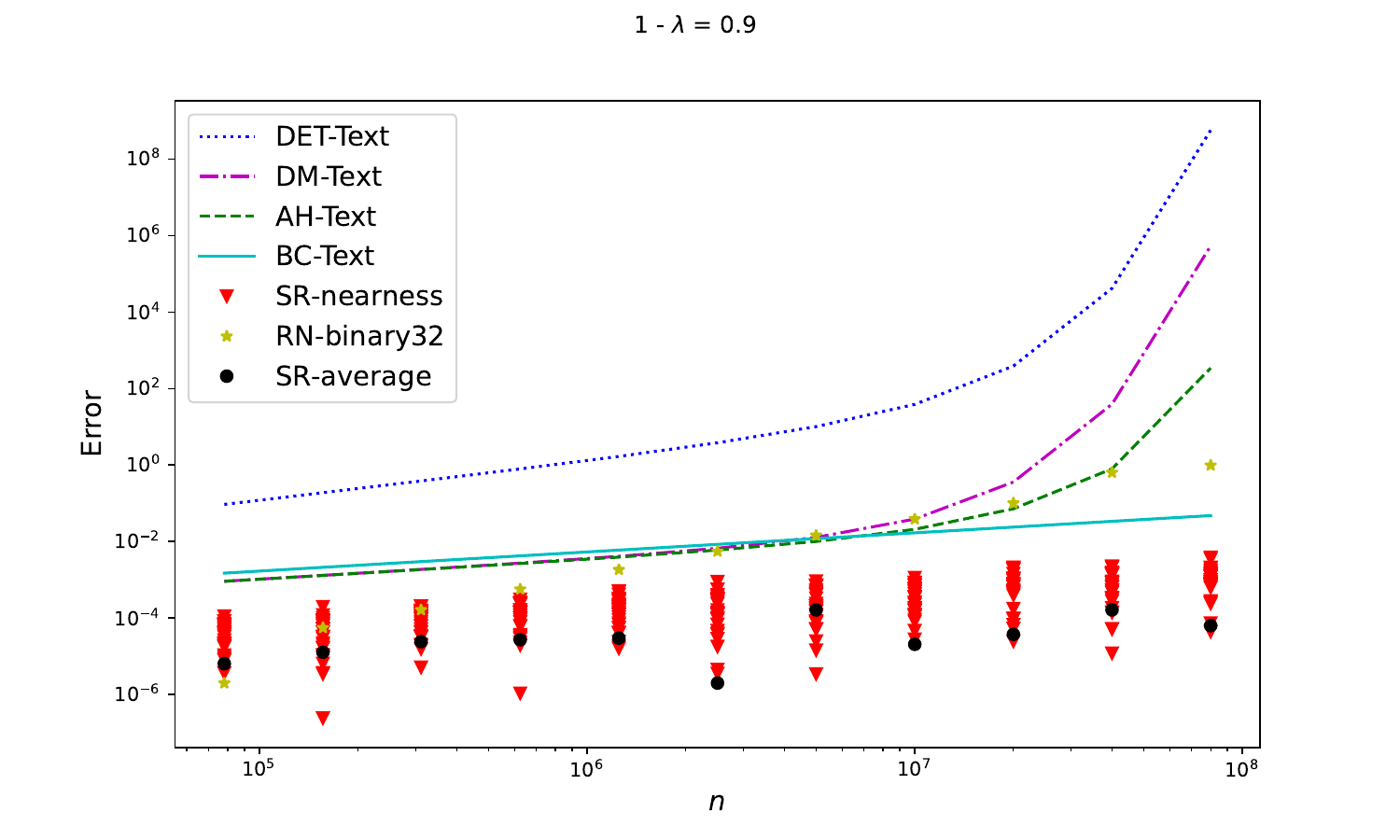}
    } \subfloat{
        \hspace{-0.8cm}\includegraphics[scale=0.29]{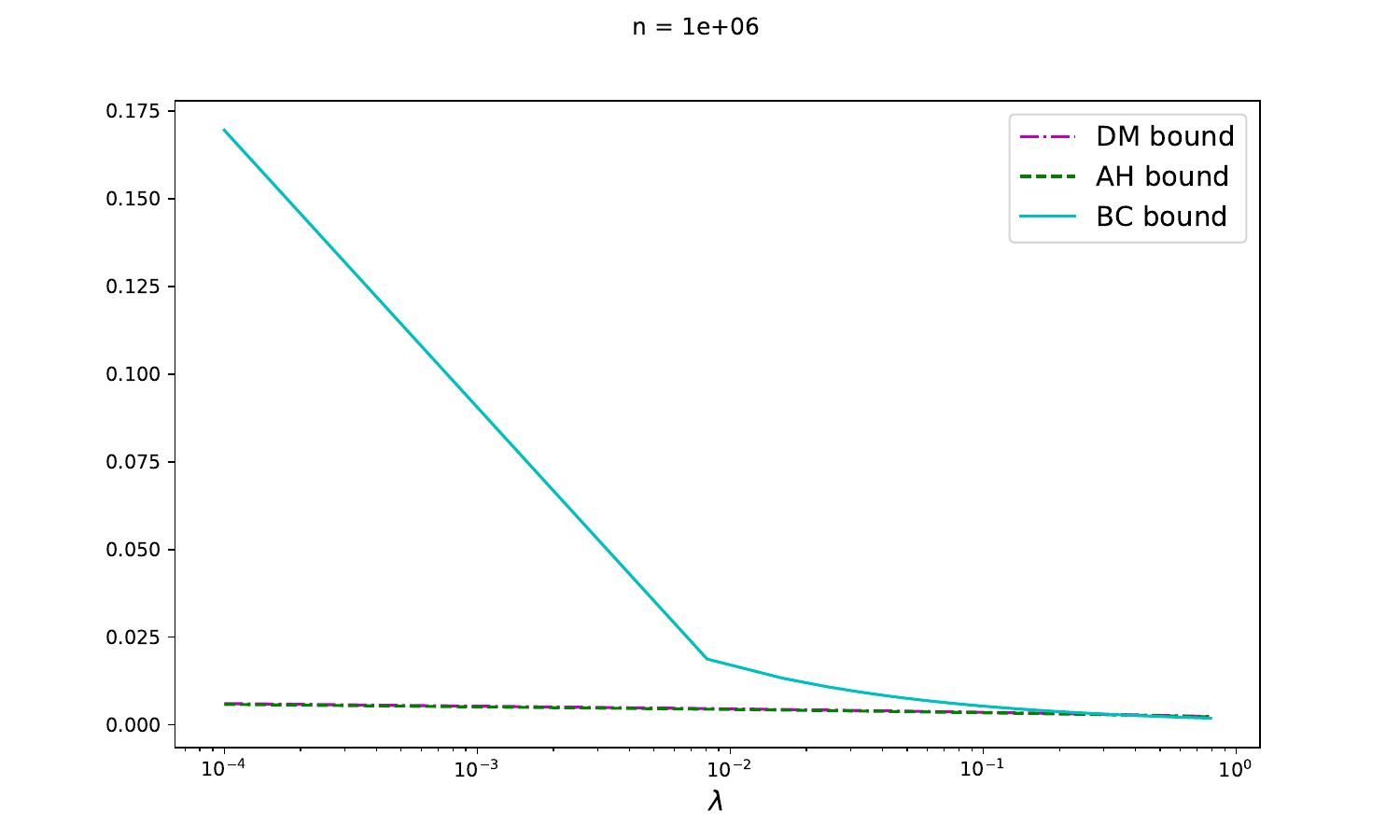}\hspace{-0.8cm}
    }
    \caption{Probabilistic error bounds over $n$ with probability $1- \lambda =0.9$ (left) and over $\lambda$ with $n=10^6$ (right) vs deterministic bound for the textbook algorithm.}
    \label{fig:textbook}
\end{figure}\vspace{-1.5cm}
\begin{figure}
    \subfloat{
        \hspace{-.6cm}\includegraphics[scale=0.29]{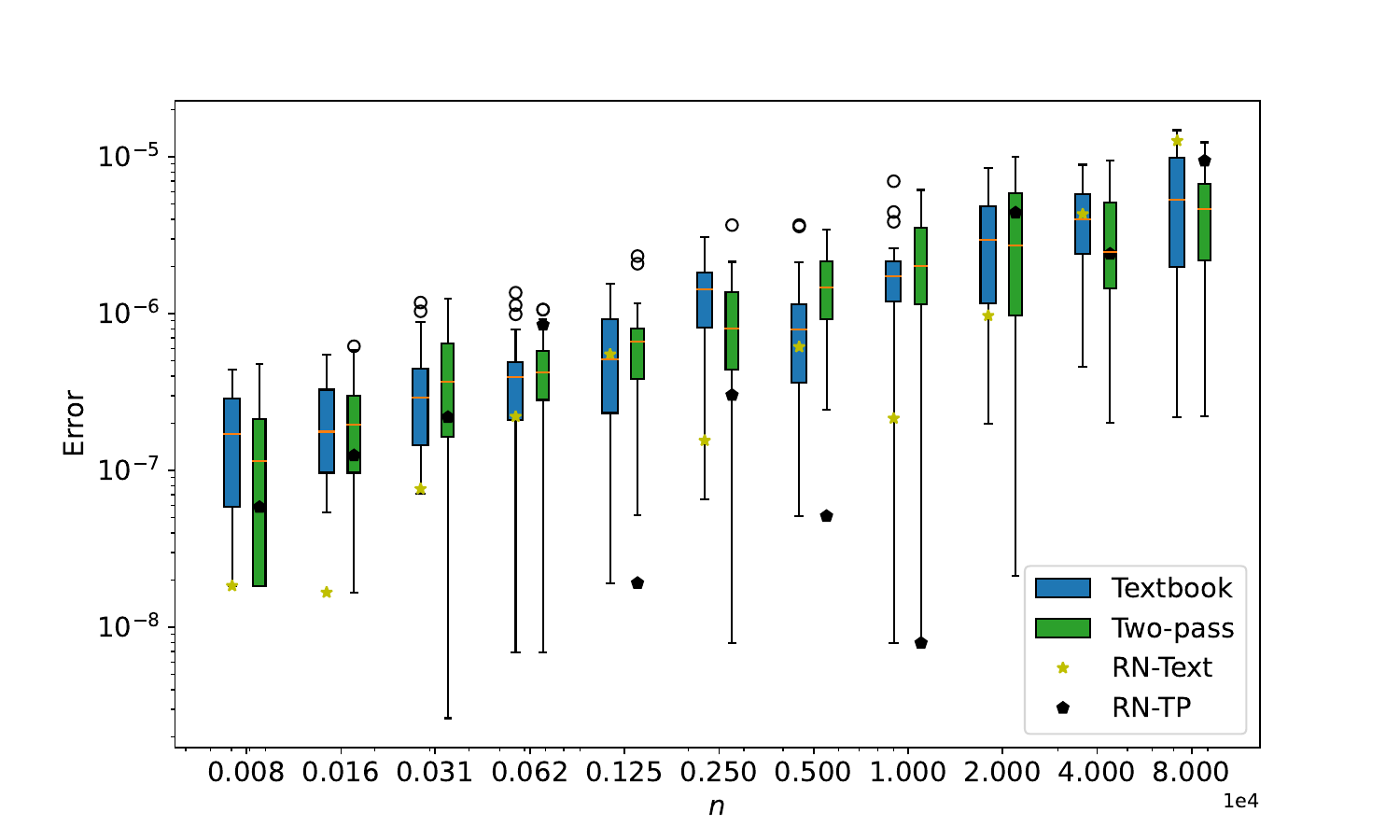}
    } \subfloat{
        \hspace{-0.8cm}\includegraphics[scale=0.29]{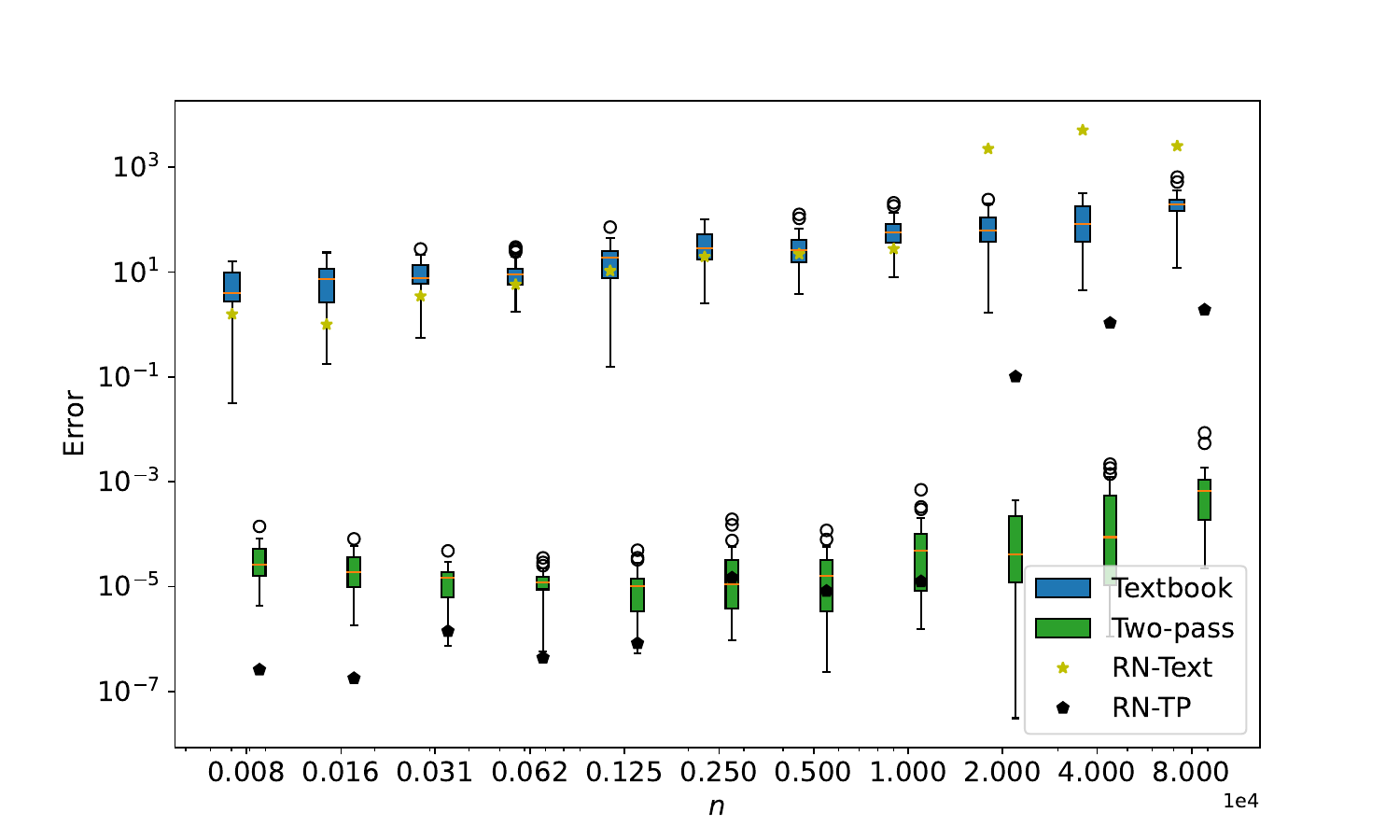}\hspace{-0.8cm}
    }
    \caption{The forward errors of textbook and two-pass algorithms in binary32 precision for floating-points chosen uniformly at random in $[-1;1]$ (left) and $[1024;1025]$ (right). }
    \label{fig:text-vs-tp}
\end{figure}

In Figure~\ref{fig:textbook}, triangles represent instances of the SR-nearness relative errors evaluation in binary32 precision, a circle marks the relative errors of the $30$ instances average, and a star represents the IEEE RN-binary32 value. Interestingly, for small $n$, the left figure shows that AH, DM, and BC bounds are comparable with a slight advantage for AH-Text and DM. However, as shown in Table~\ref{table-n}, when $nu \gg 1$, AH and DM bounds grow exponentially faster than BC bound.

As expected, for a fixed $n$, the figure on the right shows that the three bounds are close for a probability around $0.9$. Nevertheless, AH and DM bounds are more accurate for higher probabilities than BC bound.
The result is unsurprising because, generally, Azuma-Hoeffding inequality provides a bound for the deviation of the sum of a sequence of independent and bounded random variables, martingales in this instance, which gives tighter bounds for higher probabilities. In contrast, Bienaymé–Chebyshev inequality is a less restrictive result that provides an upper bound for the probability of deviation between the mean of a distribution and a particular value. The two-pass algorithm exhibits analogous boundary behavior.

\subsubsection{Textbook against two-pass}
We now compare the forward errors of both algorithms under SR. In figure~\ref{fig:text-vs-tp}, when the floating-point numbers are randomly chosen with zero mean distribution (left), the absorption errors cancel each other out because both positive and negative errors are uniformly distributed. Therefore, the computed mean is close to zero with low absolute error, and the two-pass algorithm degenerates into the textbook algorithm. Interestingly, this effect is captured by the theoretical bounds because the condition term $\mathcal{K}_2^2 + 2\mathcal{K}_1^2$ becomes smaller for zero-mean distributions. This is confirmed by the experiment in the left figure, which shows a similar forward error for the two algorithms, whether for SR or RN. 

As expected, the figure on the right illustrates that when random floating-point numbers are uniformly selected from the interval $[1024, 1025]$, the two-pass algorithm outperforms the textbook algorithm using SR or RN. The mean centering in the two-pass algorithm avoids cancellations and increases its accuracy. While the quantities $\sum_{i=1}^n x_i^2$ and $\frac{1}{n}s^2$ are inevitably very large and have the same order of magnitude, their subtraction yields a loss of significant digits in the result, which can compromise the accuracy of the textbook outcome. It is evident from this figure that the use of SR avoids stagnation for $n \geq 10^4$.

\section{Conclusion}

Many computations are non-linear in various fields such as numerical analysis. In this paper, we have chosen variance computation as an example.
In 1983, Chan, Golub, and LeVeque investigated the forward error of variance computation algorithms using RN. 
To the best of our knowledge, this is the first theoretical study of this problem using stochastic rounding as well as of any algorithm with non-linear errors. In this paper, we have presented probabilistic bounds for two variance computation algorithms that exhibit non-linear errors under SR. 

Two methods are used to estimate the forward error of computations: the BC method, which is suitable for large problem sizes $n$, and the AH method, which is preferable for higher probabilities. The study demonstrates that using SR results in probabilistic bounds on the forward error proportional to $\sqrt{n}u$, which is better than the deterministic bound in $O(nu)$ when using the default rounding mode.

While introducing pairwise algorithm in summation, textbook, and two-pass algorithms, SR leads to probabilistic bounds proportional to $\sqrt{\log(n)}u$, instead of $O(\log(n)u)$ for RN. We also demonstrate that the two-pass algorithm performs better than the textbook algorithm under SR, as it does under RN. 

A new approach based on the Doob-Meyer decomposition has been proposed as an alternative method to AH for non-linear SR computations. Our proposed approach contributes to developing new methodologies to bound the algorithms forward error under SR. Though asymptotically in $n$, this approach is equivalent to the previous two methods, we believe that it can be extended to other algorithms.

The scripts for reproducing the numerical experiments in this paper are published in the repository \url{https://github.com/verificarlo/sr-non-linear-bounds}.

\section*{Acknowledgments}
This research was supported by the InterFLOP (ANR-20-CE46-0009) project of the French National Agency for Research (ANR).

\bibliographystyle{siamplain}
\bibliography{references}

\end{document}